\documentclass[conference,letterpaper,10pt]{IEEEtran}

\usepackage[utf8]{inputenc} 
\usepackage[T1]{fontenc}
\usepackage{url}
\usepackage{ifthen}
\usepackage{cite}
\usepackage[cmex10]{amsmath} 
\usepackage{mathrsfs}
\usepackage{amssymb}
\usepackage{amsfonts}
\usepackage{dsfont}
\usepackage{float}
\usepackage{amsthm}
\usepackage{graphicx}
\usepackage{epstopdf}
\usepackage{enumerate}
\usepackage{listings}
\usepackage{lipsum}
\usepackage{multirow}
\usepackage{mathtools}
\usepackage{chemarrow}
\usepackage{extarrows}
\usepackage{bbm}
\usepackage{bm}
\usepackage{hyperref}
\usepackage{color}

\newif\iffull

\fulltrue

\renewcommand{\d}{\mathrm{d}}

\newcommand{\E}{\mathrm{E}}

\newcommand{\bx}{\mathbf{x}}
\newcommand{\bX}{\mathbf{X}}
\newcommand{\bM}{\mathbf{M}}
\newcommand{\bQ}{\mathbf{Q}}
\newcommand{\bY}{\mathbf{Y}}

\newtheorem{thm}{Theorem}
\newtheorem{lemma}{Lemma}
\newtheorem{claim}{Claim}

\newtheorem{observation}{Observation}
\newtheorem{cor}{Corollary}
\theoremstyle{definition}
\newtheorem{defn}{Definition}
\newtheorem{assump}{Assumption}
\newtheorem{remark}{Remark}

\newcommand{\T}{\lfloor c\log n\rfloor}

\mathtoolsset{showonlyrefs}

\allowdisplaybreaks

\begin{document}
\title{Derandomized Load Balancing using Random Walks on Expander Graphs} 


\author{%
  \IEEEauthorblockN{Dengwang Tang and Vijay G. Subramanian}
  \IEEEauthorblockA{University of Michigan, Ann Arbor\\
                    1301 Beal Ave, Ann Arbor, Michigan\\
                    Email: \{dwtang, vgsubram\}@umich.edu}
}


\maketitle


\begin{abstract}
	In a computing center with a large number of machines, when a job arrives, a dispatcher need to decide which machine to route this job to based on limited information. A classical method, called the \emph{power-of-$d$ choices} algorithm, is to pick $d$ servers independently at random and dispatch the job to the least loaded server among the $d$ servers. In this paper, we analyze a low-randomness variant of this dispatching scheme, where $d$ queues are sampled through $d$ independent non-backtracking random walks on a $k$-regular graph $G$. Under certain assumptions on the graph $G$, we show that under this scheme, the dynamics of the queuing system converges to the same deterministic ordinary differential equation (ODE) system for the \emph{power-of-$d$ choices} scheme. We also show that the system is stable under the proposed scheme, and the stationary distribution of the system converges to the fixed point of the ODE system. 
\end{abstract}


\section{Introduction}
In computing centers where an extremely large amount of computations are performed, there are usually a multitude of servers. This enables the computing center to handle multiple jobs at the same time. When a computational job is given to a computing center, a router, or a dispatcher, decides on which server to send the job to. The objective of the router is to minimize the queuing delay, hence enhancing the performance of the system. In several queuing system settings, it has been known that Join-the-Shortest-Queue (JSQ) is an optimal dispatching policy \cite{foschini1978basic}. However, it is not always practical to implement this policy, especially when the system contains a large number of servers, since JSQ requires the dispatcher to inquire every server's queue length, and a decision can only be made after all servers have returned the queue length information to the dispatcher. Inspired by this challenge, researchers have been analyzing schemes where the dispatcher only inquire the queue length of a small subset of servers \cite{mitzenmacher1996load, mitzenmacher2001power, vvedenskaya1996queueing}. It turns out that, in a large system, the \emph{power-of-2-choices} scheme \cite{mitzenmacher2001power}, where the job is sent to the shorter queue of 2 uniformly randomly and independently chosen servers/queues, can reduce the queuing delay significantly when comparing to the \emph{random assignment} scheme; where a job is simply sent to a randomly chosen server. The analysis is achieved through the method of \emph{fluid limit estimation}, where the evolution of the queuing system is shown to be approximately following the solution to a system of ordinary differential equations for large systems. This scheme is also extended to \emph{power-of-$d$-choices} scheme where $d$ servers are sampled for each job. The parameter $d\geq 2$ can either be a constant or grow with the size of the system. Subsequent work has been proposing and analyzing variants of the \emph{power-of-$d$-choices} scheme. The authors of \cite{mitzenmacher2002load} proposed a model where the dispatcher has memory which can store the identity of a sampled server. In this scheme, at each time the queue lengths of $d$ randomly sampled servers and the server in the memory are compared; the job joins the shortest queue among them, and the identity of the shortest queue among the $d$ random choices is saved in the memory for next job. The fluid limit approximation result for this scheme is established in \cite{luczak2013averaging}. Ying et al. \cite{ying2015power} extended the use of \emph{power-of-$d$-choices} scheme to the case of batch job arrivals, where each arrival consists of $k$ parallel tasks, slightly more than one server per task are sampled and the $k$ tasks are assigned to the sampled servers in a water-filling manner. Mukherjee et al. \cite{mukherjee2016universality} and Budhiraja et al. \cite{budhiraja2017supermarket} analyzed a variant of \emph{power-of-$d$-choices} where the servers are assumed to be interconnected through a high-degree graph, and $d$ random servers are obtained by choose a random vertex and a subset of its neighbors. Ganesh et al. \cite{ganesh2010load} also utilized an underlying graph to for load balancing problems, where jobs can switch to other queues after assignment.
In the model of ball-in-bins, where $m$ balls are placed into $n$ bins sequentially according to some policy, and the balls does not leave the bins, the \emph{power-of-$d$-choices} scheme and its variants are also analyzed. See \cite{tang2018balanced} for a list of references.

In most of the papers discussed above, the models assume that the sampling of servers for different jobs is performed in an independent manner, although dependence of the $d$ servers sampled by the same job can be present. When implementing these models, $\Theta(\log n)$ bits of randomness are required for each job. True randomness is an important resource on a computer, hence in many computer science applications, it is desirable to have a randomized algorithm which uses only a small amount of randomness. Random walks on expander graphs have been utilized to replace independent uniform randomness in many randomized algorithms \cite{hoory2006expander}. Alon et al. \cite{alon2007non} analyzed the non-backtracking random walk (NBRW) on high girth expander graphs. A typical NBRW sample path has several statistics that are similar to that of independent uniform sampling. Motivated by these works, we proposed the following variant of \emph{power-of-$d$-choices} scheme: Assuming that the servers are interconnected by a $k$-regular graph $G$, at each time a job arrives at the system, $d$ candidate servers are chosen by the location of $d$ independent non-backtracking random walkers, the job joins the shortest queue among the $d$ queues of the candidate servers, and each random walker moves independently to one of the $k-1$ neighbors and uniformly at random. We refer to this scheme as \emph{Non-backtracking Random Walk based Power-of-$d$-choices} (NBRW-Po$d$) scheme. In this paper, we analyze the NBRW-Po$d$ scheme in the standard light traffic model. The NBRW-Po$d$ scheme can be viewed as a derandomized version of \emph{power-of-$d$-choices} scheme, since one of the results of this paper is that it achieves the same performance of \emph{power-of-$d$-choices} while reducing randomness.

Our work in \cite{tang2018balanced} is closely related to this work, as the same dispatching scheme (NBRW-Po$d$) is analyzed in the balls-in-bins model. While the results in ball-in-bins model suggests that NBRW-Po$d$ has a similar behaviour as \emph{power-of-$d$-choices}, it is not clear that this is still true in the dynamic queuing system settings.

Our work in \cite{sigm} is also closely related to this work, where the  \emph{Non-backtracking Random Walk with Restart based Power-of-$d$-choices} (NBRWR-Po$d$) scheme is analyzed. The key difference between the two works lies in two places. First, the random walkers in NBRWR-Po$d$ \cite{sigm} are periodically reset to independent uniform random positions, while the random walkers in NBRW-Po$d$, the scheme in this paper, never reset. Secondly, the assumptions on the underlying graph for NBRWR-Po$d$ \cite{sigm} are weaker than that of NBRW-Po$d$: NBRW-Po$d$ requires high-girth expander graphs, while NBRWR-Po$d$ \cite{sigm} requires only high-girth graphs. 

We characterize the performance of NBRW-Po$d$ scheme via the following results:

\begin{enumerate}
	\item We provide a \emph{fluid-limit approximation} for the NBRW-Po$d$ scheme in Section \ref{sec: fluidlimit}, where the dynamics of the queuing system up to a finite time $T>0$ is shown to converge to the solution of a deterministic ordinary differential equation (ODE), which is the same ODE for the regular \emph{power-of-$d$ choices} \cite{vvedenskaya1996queueing}\cite{mitzenmacher2001power}.
	
	\item We show in Section \ref{sec: stability} that, the NBRW-Po$d$ scheme stablizes the system under the assumption that the underlying graph $G$ is connected and aperiodic.
	
	\item We show an interchange of limits result in Section \ref{sec: convsta}, which states that the stationary distribution of the queuing system under NBRW-Po$d$ scheme converges to the unique fixed point of the limiting ODE.
	
	\iffull
	\item We conduct simulations in Section \ref{sec: sim} to show that the ODE can be a good approximation for the dynamics of relatively small systems. We also investigate the dynamics of the system under different underlying graphs to explore further the relationship between graph family and performance of the scheme.
	\fi
\end{enumerate}

The proof outline of our result are as follows:
\begin{enumerate}
	\item Similar to \cite{sigm}, the queue length statistics process for the NBRW-Po$d$ scheme is not a Markov Process, hence the standard Kurtz's Theorem based fluid-limit approximation \cite{ethier2009markov, draief2010epidemics} does not apply. The methods of Wormald \cite{wormald1995differential} doesn't apply either. Just as in \cite{sigm}, because of the use of random walks introduces dependence on adjacent queues, we believe that the \emph{propagation-of-chaos} method, which was used by \cite{sznitman1991topics}, cannot be applied here. Similar to \cite{sigm}, the proof of the fluid-limit approximation result is based on martingale methods and Gronwall lemma, where the main difference to \cite{sigm} is in the lack of resets, which allows for a similar decomposition into a martingale compensator with the compensator terms small between resets. A stronger characterization of the mixing properties of the random walk is then used along with mixing times replacing the resets to achieve a similar decomposition.
	
	\item For the stability result, our proof utilizes a new variant of Foster-Lyapunov Theorem which bounds the ``future one-step drift'' of the process. As in \cite{sigm}, the theorem is applied on a subprocess and then extended to the continuous time process.
	
	\item We follow the technique of \cite{ying2015power} to prove the convergence of stationary distributions, where for the uniform tail bound part, we utilizes the new variant of Foster-Lyapunov Theorem to provide a bound. 
\end{enumerate}

\iffull
This paper is organized as follows: In Section \ref{sec: model} we introduce our model and notations. In Section \ref{sec: prelim} we prove a few preliminary results. In Section \ref{sec: main},  we state and prove our main results. We provide simulation results in Section \ref{sec: sim} and conclude in Section \ref{sec: concl}. 

\fi

\section{Model}\label{sec: model}
In this paper, we analyze the proposed dispatching scheme in the following standard model: There are $n$ servers in the system. Each server is associated with a queue. Jobs arrives at the system following a Poisson process with rate $\lambda n$, where $\lambda < 1$ is a constant. When a job arrives at the system, the dispatcher send the job to one of the servers. The services times for each job at each server are \emph{i.i.d.} exponential random variables with mean 1.

The NBRW-Po$d$ scheme is defined as following: Assume that the servers are interconnected by a $k$-regular graph $G=([n], \mathcal{E})$, where $k\geq 3$ is a constant. Let $W_1, W_2,\cdots, W_d$ be $d$ independent non-backtracking random walks on $G$. When the $j$-th job arrives at the system, allocate the job to the least loaded server among the servers $W_1\{j\}, \cdots, W_d\{j\}$. Ties are broken arbitrarily. 

To ensure that the proposed scheme has good performance, we need the following assumption on the graph $G$:

\begin{defn}[Expander Graph]\label{def:expander}
	\cite{alon2007non} Let $\{G^{(n)}\}_{n}$ be a sequence of $k$-regular graphs with $n$ vertices. Let $k=\lambda_1^{(n)}\geq \lambda_2^{(n)}\geq \cdots\geq \lambda_n^{(n)}$ be the eigenvalues of the adjacency matrix of $G^{(n)}$. Define $\lambda^{(n)} = \max\{\lambda_2^{(n)}, |\lambda_n^{(n)}| \}$. $\{G^{(n)}\}$ is called an $\lambda$-expander graph sequence if the ``second largest" eigenvalue $\lambda^{(n)}$ of the adjacency matrices of $G^{(n)}$ satisfies $\displaystyle\limsup_{n\rightarrow\infty} \lambda^{(n)} \leq \lambda$ where $\lambda$ is a constant satisfying $\lambda < k$.
\end{defn}

\begin{assump}[$G$ is a High Girth Expander]\label{ass: highgirthexpander}
	The graph sequence $\{G^{(n)}\}$ is a $k$-regular $\lambda$-expander graph sequence, and the girth of $G^{(n)}$ is greater than $2\lceil \alpha\log_{k-1}n\rceil + 1$ for sufficiently large $n$, where $\alpha$ is a positive constant.
\end{assump}

\begin{remark}
	Such graphs do exist. For example, the sequence of Ramanujan graphs, called LPS graphs, in \cite{lubotzky1988ramanujan} is a sequence of $(p+1)$-regular graphs satisfying Assumption \ref{ass: highgirthexpander} with $\alpha < \frac{2}{3}$ and $\lambda= 2\sqrt{p}$.
\end{remark}

By using a non-backtracking random walk on high-girth graphs (instead of simple random walks, or random walks on small girth graphs), it is ensured that the random walkers are not likely to revisit a vertex that it has recently visited. This allows the random walkers to find queues with relatively low load and hence reducing the queuing delay.

It is known that non-backtracking random walks mix fast on expander graphs \cite{alon2007non}. Comparing to the work in \cite{sigm}, the fast mixing of NBRW-Po$d$ plays the role of resets in NBRWR-Po$d$ \cite{sigm}, which ensures that the random walkers does not spend an extended time in a small subset of servers, hence ensuring that the server resources are sufficiently used.

\iffull
\subsection{Notations}

\begin{table}[!ht]
	\begin{tabular}{ll}
		$n$&Number of servers, also the scaling parameter\\
		&for the system.\\
		$G^{(n)}$&A regular graph of $n$ vertices.\\
		$k$&Degree of graph $G^{(n)}$, which is a constant.\\
		$Q_i^{(n)}(t)$&Queue length of server $i$ at (continuous) time $t$.\\
		$Q_i^{(n)}[j]$&Queue length of server $i$ at the $j$-th event\\
		&(i.e. arrival/potential service)\\
		$Q_i^{(n)}\{j\}$&Queue length of server $i$ as seen by the $j$-th arrival job\\
		$X_i^{(n)}(t)$&Proportion of servers with load at least $i$ at time $t$.\\
		$W_l\{j\}$&Position of the non-backtracking random walker $l$ after \\
		&$j$ steps of transition.\\
		$\overline{W}_l\{j\}$&The directed edge that points towards $W_l\{j\}$ from \\&$W_l\{j-1\}$\\
		$T_j$&Arrival time of the $j$-th job.\\
		$\tau_j$&Inter-arrival time between the $(j-1)$-th and $j$-th job.\\
		$\mathcal{F}_t$&A filtration that the random walk and queuing process \\&are adapted to 
	\end{tabular}
	\caption{Notations in this paper}\label{chart}
\end{table}

For the ease of exposition, we will drop the superscript $(n)$ in the proofs when $n$ is clear from context.

Four different brackets are employed to indicate different time index systems: $(t)$ is used for continuous time; $[j]$ indicates the time of the $j$-th arrival and potential service, or the $j$-th transition in a uniformized chain; $\{j\}$ indicates the time of the $j$-th arrival job; $\langle j\rangle$ indicates the time of the $j\T$-th arrival. 

All notations in the paper will follow Table \ref{chart}, with the exception of Section \ref{sec: prelim}, where general lemmas are proved and notations stands for general processes and variables.

For simplicity of notation, when $X$ is a random variable, $t$ is a constant, and $\mathcal{A}$ is an event, we use $\Pr(X\geq t, \mathcal{A} )$ to represent $\Pr(\{\omega: X(\omega)\geq t \}\cap \mathcal{A})$.

In this paper, ``$X\stackrel{d}{\sim}Y$'' means that ``$X$ and $Y$ have the same distribution.'' ``$\Rightarrow$" stands for weak convergence, or convergence in distribution. $d_{W_1, \|\cdot\|_1}$ stands for the $W_1$ Wasserstein distance of measures on a metric space where the metric is induced by the norm $\|\cdot\|_1$.

\fi

\section{Preliminary Results for NBRW-Po$d$}\label{sec: prelim}
In this section, we will prove a few general preliminary results for the proofs of our main results. 

\subsection{Large Deviation Results}
The following results will be some basic concentration inequalities that we will use.

\begin{lemma}[Bernstein]\label{lem: bern}
	Let $\{Z_j\}_{j=1}^\infty$ be a process adapted to the filtration $\{\mathcal{F}_j\}_{j=-1}^\infty$. Let $N> 0$ be even. If $0\leq Z_j\leq B$ and $\E[Z_j|\mathcal{F}_{j-2}]\leq m$ a.s. for all $j\geq 1$, then for any $\lambda\geq 2Nm$, we have
	\begin{equation}
	\Pr\left( \sum_{j=1}^N Z_j\geq\lambda \right)\leq 2\exp\left(-\dfrac{3\lambda}{32B}\right)
	\end{equation}
\end{lemma}

\begin{proof}
	By the Union Bound, we have
	\begin{align*}
	&\quad \Pr\left( \sum_{j=1}^N Z_j\geq\lambda \right)\\
	&\leq \Pr\left(\sum_{j=1}^{\frac{N}{2}} Z_{2j} \geq \dfrac{\lambda}{2} \right) + \Pr\left(\sum_{j=1}^{\frac{N}{2}} Z_{2j-1} \geq \dfrac{\lambda}{2} \right)
	\end{align*}
	
	Applying the Bernstein Inequality proved in \cite{tang2018balanced}, we obtain that
	\begin{align*}
	\Pr\left(\sum_{j=1}^{\frac{N}{2}} Z_{2j} \geq \dfrac{\lambda}{2} \right) &\leq \exp\left(-\dfrac{3(\lambda/2)}{16B}\right) = \exp\left(-\dfrac{3\lambda}{32B}\right)\\
	\Pr\left(\sum_{j=1}^{\frac{N}{2}} Z_{2j-1} \geq \dfrac{\lambda}{2} \right) &\leq \exp\left(-\dfrac{3(\lambda/2)}{16B}\right) = \exp\left(-\dfrac{3\lambda}{32B}\right).
	\end{align*}
	
	Combining all the above we prove the result.
\end{proof}

\subsection{A Precise Estimate of Sampling Probability}
We have a sharper characterization of the mixing property from Alon et al. \cite{alon2007non}.

\begin{lemma}\label{lem: mixing}
	Let $V^{(n)}(t)$ be a non-backtracking random walk on $k$-regular $\lambda$-expander graph $G^{(n)}$. Define $P_{u_1, v, u_0}^{(t)}:=\Pr(V^{(n)}(t+1) =v~|~V^{(n)}(0) = u_0, V^{(n)}(1) = u_1)$. Then there exists a constant $c>0$ (which only depends on $\lambda$ and $k$) such that for sufficiently large $n$,
	\begin{align}
	\max_{u_0, u_1, v\in G^{(n)}} \left|P_{u_1, v, u_0}^{(t)} - \dfrac{1}{n}\right|\leq \dfrac{1}{n^2}\qquad \forall t\geq \T
	\end{align}
	
\end{lemma}

\begin{proof}
	The proof utilizes the same observations as in \cite{alon2007non}. The result here strengthens Lemma 5 in \cite{tang2018balanced}.
	
	Let $A_{u, v}^{(t)}$ denote the number of non-backtracking walks of length $t$ from $u$ to $v$.
	Let $\tilde{P}^{(t)}$ be the $t$-step transition probability matrix of a non-backtracking random walk, where the first step of the random walk is to a uniform random neighbor of the starting vertex. We have $\tilde{P}^{(t)} = \frac{A^{(t)}}{k(k-1)^{t-1}}$.
	
	From the proof of Lemma 5 in \cite{tang2018balanced}, we have the estimate
	\begin{align}
	&\quad\max_{u, v}\left|\tilde{P}_{u,v}^{(t)} - \dfrac{1}{n} \right|\\
	&\leq \dfrac{k-1}{k}(t+1)\beta^t + \dfrac{1}{k(k-1)}(t-1)\beta^{t-2}\qquad t\geq 2
	\end{align}
	where $\frac{1}{\sqrt{k-1}}<\beta < 1$ is a constant which depends on $\lambda$ and $k$.
	
	We immediately obtain
	\begin{align}
	&\quad \max_{u, v}\left|A_{u,v}^{(t)} - \dfrac{k(k-1)^{t-1}}{n} \right|\\
	&\leq (t+1)[(k-1)\beta]^t +(t-1)[(k-1)\beta]^{t-2}\qquad t\geq 2
	\end{align}
	
	Our precise estimate will be based on the following observation: Let $A_{u_1, v, u_0}^{(t)}$ denote the number of non-backtracking walks of length $t$ from $u_1$ to $v$ which avoid $u_0$ in the first step. By a counting argument, we establish
	\begin{equation}
	A_{u_1,v, u_0}^{(t)} = A_{u_1, v}^{(t)} - A_{u_0, v, u_1}^{(t-1)}
	\end{equation}
	
	Applying the above observation iteratively, we obtain
	\begin{align*}
	A_{u_1,v, u_0}^{(t)} &= A_{u_1, v}^{(t)} - A_{u_0, v}^{(t-1)} + A_{u_1, v}^{(t-2)} - A_{u_0, v}^{(t-3)} \\&\quad + \cdots + (-1)^t A_{u_b, v}^{(2)} + (-1)^{t+1} A_{u_{1-b}, v, u_b}^{(1)}
	\end{align*}
	where $b=1$ if $t$ is even, and $b=0$ otherwise.
	
	We have $|(-1)^{t+1} A_{u_{1-b}, v, u_b}^{(1)}|\leq 1$, hence by the triangle inequality,
	\begin{align*}
	&\quad \left| A_{u_1,v, u_0}^{(t)} - \sum_{j=0}^{t-2} (-1)^{j} \dfrac{k(k-1)^{t-1-j}}{n} \right|\\
	&\leq \sum_{j=0}^{t-2}\left|A_{u_{b(j)}, v}^{(t-j)} - \dfrac{k(k-1)^{t-1-j}}{n} \right| + 1
	\end{align*}
	where $b(j)=1$ if $j$ is even, and $b(j)=0$ otherwise.
	
	Hence we have
	\begin{equation}\label{eqrwprob}
	\begin{split}
	&\quad \max_{u_0, u_1, v}\left| A_{u_1,v, u_0}^{(t)} - \sum_{j=0}^{t-2} (-1)^{j} \dfrac{k(k-1)^{t-1-j}}{n} \right|\\
	&\leq \sum_{j=0}^{t-2}  (t-j+1)[(k-1)\beta]^{t-j} \\
	&\quad + \sum_{j=0}^{t-2} (t-j-1)[(k-1)\beta]^{t-j-2}  + 1\\
	&\leq t [(k-1)\beta]^t \sum_{j=0}^{t-2} \left( [(k-1)\beta]^{-j} + [(k-1)\beta]^{-j-2} \right) + 1\\
	&\leq t [(k-1)\beta]^t \sum_{j=0}^{t-2} \left( (\sqrt{k-1})^{-j} + (\sqrt{k-1})^{-j-2} \right) + 1\\
	&\leq 2t(t-1) [(k-1)\beta]^t + 1
	\end{split}
	\end{equation}
	
	We compute
	\begin{equation}
	\begin{split}
	&\quad \sum_{j=0}^{t-2} (-1)^{j} \dfrac{k(k-1)^{t-1-j}}{n} \\
	&= \dfrac{k(k-1)^{t-1}}{n}\cdot  \dfrac{1-(-k+1)^{-t+1}}{1-(-k+1)^{-1}} \\
	&= \dfrac{(k-1)^{t}}{n}[1-(-k+1)^{-t+1}]
	\end{split}
	\end{equation}
	
	Denote $P_{u_1,v, u_0}^{(t)} = \Pr(V^{(n)}(t+1) =v~|~V^{(n)}(0) = u_0, V^{(n)}(1) = u_1)$. Dividing both sides of \eqref{eqrwprob} by $(k-1)^t$ we obtain
	\begin{align*}
	\max_{u_0, u_1, v}\left| P_{u_1,v, u_0}^{(t)} - \dfrac{1-(-k+1)^{-t+1}}{n} \right|\leq  2t(t-1) \beta^t + (k-1)^{-t}
	\end{align*}
	
	Therefore
	\begin{align}
	&\quad \max_{u_0, u_1, v}\left| P_{u_1,v, u_0}^{(t)} - \dfrac{1}{n} \right|\\
	&\leq  2t(t-1) \beta^t + (k-1)^{-t} + \left| \dfrac{(-k+1)^{-t+1}}{n}  \right| \\
	&= 2t(t-1) \beta^t + \left(1+\dfrac{1}{n}\right)(k-1)^{-t+1}\\
	&\leq \left[2t(t-1) + 2(k-1)\right] \beta^t\qquad \qquad \text{(Since $\beta \geq \dfrac{1}{k-1}$)}
	\end{align}
	
	The rest of the proof is similar to that of Lemma 5 in \cite{tang2018balanced}. Observe that RHS of above is decreasing in $t$ for sufficiently large $t$. Pick $c = -\frac{3}{\log \beta }$ and set $\tau = \lfloor c\log n\rfloor$, for sufficiently large $n$ we have
	\begin{align*}
	&\quad \max_{u_0, u_1, v}\left| P_{u_1,v, u_0}^{(t)} - \dfrac{1}{n} \right|\\
	&\leq \left[2\tau(\tau-1) + 2(k-1)\right] \beta^\tau = O\left(\dfrac{(\log n)^2}{n^3}\right)\qquad \forall t\geq \T
	\end{align*}
	
	Hence, for sufficiently large $n$, we have
	\begin{align*}
	\max_{u_0, u_1, v}\left| P_{u_1,v, u_0}^{(t)} - \dfrac{1}{n} \right|&\leq \dfrac{1}{n^2}\qquad\qquad  \forall t\geq \lfloor c\log n\rfloor 
	\end{align*}
\end{proof}
\subsection{A Variant of the Foster-Lyapunov Theorem}
The following extensions of the Foster-Lyapunov Theorem will also be used.

\begin{lemma}\label{lem: tsfl}
	An irreducible Markov Chain $\{X_j\}_{j\in\mathbb{N}}$ is positive recurrent if there exists a function $V:\mathcal{S}\mapsto \mathbb{R}_+$, positive integers $1\leq L< K$, and a finite set $B\subset \mathcal{S}$ satisfying the following conditions:
	\begin{align}
	&\E[V(X_{k+1})] < +\infty\quad \text{whenever}\quad \E[V(X_k)] < +\infty\\
	&\E[V(X_{k+K}) - V(X_{k+L})~|~X_k=x] \leq -\epsilon + A\mathbbm{1}_{ B }(x)\label{negdrift}
	\end{align}
	for some $\epsilon>0$ and $A<+\infty$.
\end{lemma}

\begin{proof}
	WLOG, assume that the set $B$ is non-empty.
	
	Let $\mathcal{F}_k:=\sigma(X_0, X_1,\cdots,X_k)$. 
	Define $\tau :=\inf \{t\in\mathbb{N}~|~X_t\in B \}$. $\tau$ is a stopping time with respect to the filtration $(\mathcal{F}_t)_{t\in\mathbb{N}}$.
	
	Fix $x_0\in B$. Start the chain from $X_0=x_0$. We have $\E[V(X_0)]=V(x_0) < +\infty$. By assumption, we then have $\E[V(X_k)] < +\infty$ for all $k\in\mathbb{N}$. As a consequence, $\E[V(X_{l})~|~\mathcal{F}_k]<+\infty$ a.s. for all $l,k\in\mathbb{N}$.
	
	Using \eqref{negdrift}, we obtain
	\begin{align*}
	\E[V(X_{k+K})~|~\mathcal{F}_k] + \epsilon \leq \E[V(X_{k+L})~|~\mathcal{F}_k] + A \mathbbm{1}_{B}(X_k)
	\end{align*}
	
	Multiply both sides by $\mathbbm{1}_{\{\tau > k \}}$, using the fact that $\mathbbm{1}_{\{\tau > k \}}$ is $\mathcal{F}_k$-measurable, we obtain
	\begin{align*}
	&\quad \E[V(X_{k+K})\mathbbm{1}_{\{\tau > k \}}~|~\mathcal{F}_k] + \epsilon \mathbbm{1}_{\{\tau > k \}} \\&\leq \E[V(X_{k+L}) \mathbbm{1}_{\{\tau > k \}}~|~\mathcal{F}_k]  + A \mathbbm{1}_{B}(X_k) \mathbbm{1}_{\{\tau > k \}}
	\end{align*}
	
	Using $\mathbbm{1}_{\{\tau > k \}} \geq \mathbbm{1}_{\{\tau > k + K - L \}}$, and the fact that $\mathbbm{1}_{B}(X_k) \mathbbm{1}_{\{\tau > k \}} = \mathbbm{1}_{\{k=0\}}$, we have
	\begin{align*}
	&\quad \E[V(X_{k+K})\mathbbm{1}_{\{\tau > k+K-L \}}~|~\mathcal{F}_k] + \epsilon \mathbbm{1}_{\{\tau > k \}} \\&\leq \E[V(X_{k+L}) \mathbbm{1}_{\{\tau > k \}}~|~\mathcal{F}_k]  + A \mathbbm{1}_{\{k=0\}}
	\end{align*}
	
	Taking expectation of both sides, we obtain
	\begin{align*}
	&\quad \E[V(X_{k+K})\mathbbm{1}_{\{\tau > k+K-L \}}] + \epsilon \Pr(\tau > k )\\ &\leq \E[V(X_{k+L}) \mathbbm{1}_{\{\tau > k \}}]  + A \mathbbm{1}_{\{k=0\}}
	\end{align*}
	
	Let $m\in\mathbb{N}$, summing both sides over $k=0,1,\cdots, m$ we obtain
	\begin{align*}
	&\quad \sum_{k=K-L}^{m+K-L}\E[V(X_{k+L})\mathbbm{1}_{\{\tau> k \}}] + \epsilon \sum_{k=0}^m\Pr(\tau > k ) \\
	&\leq \sum_{k=0}^{m}\E[V(X_{k+L}) \mathbbm{1}_{\{\tau > k \}}]  + A
	\end{align*}
	
	Every term in the above inequality is finite, hence we can rearrange the inequality to obtain
	\begin{align*}
	&\quad \epsilon \sum_{k=0}^m\Pr(\tau > k ) \\
	&\leq A + \sum_{k=0}^{K-L-1}\E[V(X_{k+L}) \mathbbm{1}_{\{\tau > k \}}] \\
	&\qquad - \sum_{k=m+1}^{m+K-L}\E[V(X_{k+L}) \mathbbm{1}_{\{\tau > k \}}]\\
	&\leq A + \sum_{k=0}^{K-L-1}\E[V(X_{k+L}) \mathbbm{1}_{\{\tau > k \}}]
	\end{align*}
	where the last inequality is true since $V(x)\geq 0$ for all $x\in S$.
	
	Take $m\rightarrow\infty$, we obtain
	\begin{align*}
	\epsilon \E[\tau] = \epsilon \sum_{k=0}^\infty \Pr(\tau > k ) \leq  A + \sum_{k=0}^{K-L-1}\E[V(X_{k+L})] < +\infty
	\end{align*}
	
	Hence starting from $x_0\in B$, the expected hitting time of $B$ is finite. Using Lemma 2.1.3 from \cite{hajek2006notes} we conclude that $X_j$ is positive recurrent.
	
\end{proof}

\begin{lemma}[Moment Bound]\label{lem: momentbound}
	Suppose $\{X_j\}_{j\in\mathbb{N}}$ is positive recurrent, $V, f, g$ are non-negative functions on $\mathcal{S}$, $1\leq L < K$, and suppose
	\begin{align}
	&\E[V(X_{k+1})] < +\infty\quad \text{whenever}\quad \E[V(X_k)] < +\infty\\
	&\E[V(X_{k+K}) - V(X_{k+L})~|~X_k=x] \leq -f(x) + g(x)\quad \forall x\in \mathcal{S}\label{negdrift1}
	\end{align}
	
	Let $\hat{X}$ have the same distribution as the stationary distribution of $\{X_j\}_{j\in\mathbb{N}}$. Then $\E[f(\hat{X})]\leq \E[g(\hat{X})]$.
\end{lemma}

\begin{proof}
	Let $\mathcal{F}_k:=\sigma(X_0, X_1,\cdots, X_k)$. 
	
	Let $x_0\in\mathcal{S}$ be fixed. Set $X_0\equiv x_0$. 	
	We have $\E[V(X_0)]=V(x_0) < +\infty$. By assumption, we then have $\E[V(X_k)] < +\infty$ for all $k\in\mathbb{N}$. As a consequence, $\E[V(X_{l})~|~\mathcal{F}_k]<+\infty$ a.s. for all $l,k\in\mathbb{N}$. Hence
	\begin{equation}
	\E[V(X_{k+K})~|~\mathcal{F}_k] + f(X_k) \leq \E[V(X_{k+L})~|~\mathcal{F}_k] + g(X_k)
	\end{equation}
	
	Let $\tau$ be any stopping time w.r.t. $\{\mathcal{F}_k\}_{k=0}^\infty$. Multiply both sides by $\mathbbm{1}_{\{\tau > k \}}$, using the fact that $\mathbbm{1}_{\{\tau > k \}}$ is $\mathcal{F}_k$-measurable, we obtain
	\begin{align*}
	&\quad \E[V(X_{k+K})\mathbbm{1}_{\{\tau > k \}}~|~\mathcal{F}_k] + f(X_k) \mathbbm{1}_{\{\tau > k \}} \\
	&\leq \E[V(X_{k+L}) \mathbbm{1}_{\{\tau > k \}}~|~\mathcal{F}_k]  + g(X_k) \mathbbm{1}_{\{\tau > k \}}
	\end{align*}
	
	Using $\mathbbm{1}_{\{\tau > k \}} \geq \mathbbm{1}_{\{\tau > k + K - L \}}$, we have
	\begin{align*}
	&\quad \E[V(X_{k+K})\mathbbm{1}_{\{\tau > k+K-L \}}~|~\mathcal{F}_k] + f(X_k) \mathbbm{1}_{\{\tau > k \}} \\
	&\leq \E[V(X_{k+L}) \mathbbm{1}_{\{\tau > k \}}~|~\mathcal{F}_k]  + g(X_k) \mathbbm{1}_{\{\tau > k \}}
	\end{align*}
	
	Taking expectation of both sides, we obtain
	\begin{align*}
	&\quad \E[V(X_{k+K})\mathbbm{1}_{\{\tau > k+K-L \}}] + \E[f(X_k)\mathbbm{1}_{\{\tau >k\} }] \\
	&\leq \E[V(X_{k+L}) \mathbbm{1}_{\{\tau > k \}}]  + \E[g(X_k)\mathbbm{1}_{\{\tau >k\} }]
	\end{align*}
	
	Let $n\in\mathbb{N}$, summing both sides over $k=0,1,\cdots, n$ we obtain
	\begin{align*}
	&\quad \sum_{k=K-L}^{n+K-L}\E[V(X_{k+L})\mathbbm{1}_{\{\tau> k \}}] + \sum_{k=0}^n\E\left[ f(X_k)\mathbbm{1}_{\{\tau >k\} }\right] \\
	&\leq \sum_{k=0}^{n}\E[V(X_{k+L}) \mathbbm{1}_{\{\tau > k \}}]  + \sum_{k=0}^n\E\left[ g(X_k)\mathbbm{1}_{\{\tau >k\} }\right] 
	\end{align*}
	
	Every term in the above inequality is finite, hence we can rearrange the inequality to obtain
	\begin{align*}
	&\quad \sum_{k=0}^n\E\left[ f(X_k)\mathbbm{1}_{\{\tau >k\} }\right] \\
	&\leq \sum_{k=0}^n\E\left[ g(X_k)\mathbbm{1}_{\{\tau >k\} }\right] + \sum_{k=0}^{K-L-1}\E[V(X_{k+L}) \mathbbm{1}_{\{\tau > k \}}] \\
	&\qquad - \sum_{k=m+1}^{m+K-L}\E[V(X_{k+L}) \mathbbm{1}_{\{\tau > k \}}]\\
	&\leq \sum_{k=0}^n\E\left[ g(X_k)\mathbbm{1}_{\{\tau >k\} }\right] + \sum_{k=0}^{K-L-1}\E[V(X_{k+L}) \mathbbm{1}_{\{\tau > k \}}]
	\end{align*}
	where the last inequality is true since $V(x)\geq 0$ for all $x\in S$.
	
	Take $n\rightarrow\infty$, we obtain
	\begin{align*}
	\E\left[ \sum_{k=0}^{\tau-1} f(X_k) \right] &\leq \E\left[ \sum_{k=0}^{\tau-1} g(X_k) \right] + \sum_{k=0}^{K-L-1}\E[V(X_{k+L})]
	\end{align*}
	
	Let $T_m$ be the time of the $m$-th return to state $x_0$. $T_m$ is a stopping time w.r.t. $\{\mathcal{F}_k\}_{k=0}^\infty$. Hence
	\begin{align*}
	\E\left[ \sum_{k=0}^{T_m-1} f(X_k) \right] &\leq \E\left[ \sum_{k=0}^{T_m-1} g(X_k) \right] + \sum_{k=0}^{K-L-1}\E[V(X_{k+L})]
	\end{align*}
	
	Using the equality of time and statistical averages we have 
	\begin{align*}
	\E\left[ \sum_{k=0}^{T_m-1} f(X_k) \right] &= m\E[T_1] \E[f(\hat{X})]\\
	\E\left[ \sum_{k=0}^{T_m-1} g(X_k) \right] &= m\E[T_1] \E[g(\hat{X})]\\
	\end{align*}
	
	Hence we have
	\begin{align*}
	m\E[T_1] \E[f(\hat{X})] \leq \sum_{k=0}^{K-L-1}\E[V(X_{k+L})] + m\E[T_1] \E[g(\hat{X})]
	\end{align*}
	where $\sum_{k=0}^{K-L-1}\E[V(X_{k+L})] < +\infty$. Dividing both sides by $m\E[T_1]$ and let $m\rightarrow\infty$ we obtain the result.
\end{proof}

\section{Main Results}\label{sec: main}
\subsection{Fluid Limit Approximation for NBRW-Po$d$}\label{sec: fluidlimit}

In this section, we prove our first main result: the queuing system dynamics under NBRW-Po$d$ scheme for a large system can be approximated by the solutions to a system of differential equations, which is the same ODE as that of \emph{power-of-$d$} scheme. We assume that $G^{(n)}$ satisfies assumption \ref{ass: highgirthexpander} throughout this section.

\begin{thm}\label{thm: main}
	Consider the dynamic system $\bx(t)\in [0, 1]^{\mathbb{Z}_+}$ described by the following differential equations:
	\begin{equation}\label{diffeq}
	\begin{split}
	\dfrac{\d x_i}{\d t} &= \lambda (x_{i-1}^d - x_i^d) - (x_i-x_{i+1})\qquad i\geq 1\\
	x_0(t) &\equiv 1
	\end{split}
	\end{equation}
	
	Let $\bX^{(n)}(t) = (X_i^{(n)}(t))_{i\in\mathbb{Z}_+}$ be an infinite dimensional vector, where $X_i^{(n)}(t)$ is the proportion of queues with length exceeding (or equal to) $i$ at time $t$. Suppose that
	\begin{enumerate}[(a)]
		\item Random walkers are initialized to independent uniform random positions.
		\item $\bQ^{(n)}(0)$ is deterministic
		\item $\displaystyle\lim_{n\rightarrow\infty}\|\bX^{(n)}(0) - \bx(0)\|_1 = 0$
		\item $\|\bx(0)\|_1<+\infty$
	\end{enumerate} 
	then for every finite $T>0$
	\begin{equation}
	\lim_{n\rightarrow\infty} \sup_{0\leq t\leq T}\|\bX^{(n)}(t) - \bx(t)\|_1 = 0\qquad a.s.
	\end{equation}
\end{thm}

\begin{proof}[Proof of Theorem \ref{thm: main}]
	For the inequalities we proved in all the proofs in this section, the inequality should be understood as true for all sufficiently large $n$. If not explicitly stated, then the threshold for sufficiently large $n$ depends only on the system parameters (i.e. $d, k, \alpha$, and $\lambda$).
	
	Let $\mathbf{a}, \mathbf{b}: [0, 1]^{\mathbb{Z}_+}\mapsto [0, 1]^{\mathbb{Z}_+}$ be defined as
	\begin{align*}
	a_i(\bx) &= \begin{cases}
	\lambda(x_{i-1}^d - x_i^d)&i\geq 1\\
	0&i=0
	\end{cases}\\
	b_i(\bx) &= \begin{cases}
	x_i-x_{i+1}&i\geq 1\\
	0&i=0
	\end{cases}
	\end{align*}
	
	Here $\mathbf{a}(\bx)+\mathbf{b}(\bx)$ will be the mean field transition rate for power of $d$-scheme. Here, we separate the analysis for arrival and departure parts.  Both $\mathbf{a}$ and $\mathbf{b}$ are Lipschitz continuous operators with respect to $\ell_1$ norm, where $\mathbf{a}$ has Lipschitz constant $2\lambda d$ and $\mathbf{b}$ has Lipschitz constant $2$.
	
	For $i\geq 1$, define $A_i(t)$ to be the total number of arrival jobs that are dispatched to a server with load $i-1$ (just before this arrival) before (including) time $t$. Define $A_0(t) \equiv 0$. Also define $B_i(t)$ to be the number of departures from queues with load $i$ (just before this departure) before (including) time $t$. Let $\mathbf{A}(t), \mathbf{B}(t)$ denote the corresponding infinite dimensional vectors. We have the relation
	\begin{equation}
	\bX(t) = \bX(0) + \dfrac{\mathbf{A}(t)}{n} - \dfrac{\mathbf{B}(t)}{n}
	\end{equation}
	
	Now, define
	\begin{equation}
	\bM(t) = \bX(t) - \bX(0) - \int_0^t [\mathbf{a}(\bX(u)) - \mathbf{b}(\bX(u))]\d u
	\end{equation}
	
	The idea of the proof is to bound $\|\bM(t)\|_1$ and then apply Gronwall's lemma to bound $\|\bX(t) - \bx(t)\|_1$. We write
	\begin{align*}
	\bM(t) &= \left[\dfrac{\mathbf{A}(t)}{n} - \int_0^t \mathbf{a}(\bX(u))\d u\right] - \left[\dfrac{\mathbf{B}(t)}{n}-\int_0^t \mathbf{b}(\bX(u))\d u \right]\\
	&=:\bM^a(t) - \bM^b(t)
	\end{align*}
	where $\bM^a(t)$ is the ``arrival part", and $\bM^b(t)$ is the ``service" part. We will bound $\|\bM^a(t)\|_1$ and $\|\bM^b(t)\|_1$ separately.
	
	Now we define two auxiliary queuing processes $\tilde{\bQ}^+(t)$ and $\tilde{\bQ}^-(t)$, which are coupled with $\bQ(t)$.
	
	Recall that an alternative description of $\bQ(t)$ is as follows: $\bQ(t)$ can be obtained from a discrete-time process $\bQ[j]$ along with holding times, where the holding times are \emph{i.i.d.} exponentially distributed with mean $\frac{1}{(\lambda+1)n}$. The evolution of $\bQ[j]$ can be described as follows
	\begin{equation}
	\bQ[j+1] = (\bQ[j] + \mathbf{R}[j](1-\Lambda[j]) - \mathbf{S}[j]\Lambda[j] )_+
	\end{equation}
	where $\Lambda[j], j=0,1,2,\cdots$ are \emph{i.i.d.} Bernoulli random variables with mean $\frac{1}{\lambda+1}$, $\mathbf{S}[j]$ are i.i.d. uniformly distributed on $\{e_i\}_{i=1}^n$ which indicates potential services, $\mathbf{R}[j]\in \{e_i \}_{i=1}^n$ are potential arrivals to the system. Note that $\{\Lambda[j]\}_{j=0}^\infty$ and $\{\mathbf{S}[j]\}_{j=0}^\infty$ are mutually independent, and also independent of $\{\mathbf{Q}[j']\}_{j'=0}^j$. 
	
	In the discrete process, let $\tilde{T}[s]$ be the arrival timestamp of the $s$-th job (i.e. the $s$-th timestamp such that $\Lambda[j]=0$.) Let $c>0$ be a constant such that Lemma \ref{lem: mixing} holds. Define	 $\tilde{\bQ}^{ + }[0] = \tilde{\bQ}^{ - }[0] = \bQ[0]$ and
	\begin{equation}
	\tilde{\bQ}^{+}[j+1] = \begin{cases}
	\bQ[j]&j= \tilde{T}[2s\T] \\
	&\text{for some }s\in\mathbb{Z}_+\\
	(\tilde{\bQ}^{+}[j] - \mathbf{S}[j]\Lambda[j] )_+&\text{otherwise}
	\end{cases}
	\end{equation} 
	
	\begin{equation}
	\tilde{\bQ}^{- }[j+1] = \begin{cases}
	\bQ[j]&j = \tilde{T}[(2s+1)\T] \\
	&\text{for some }s\in\mathbb{Z}_+\\
	(\tilde{\bQ}^{- }[j] - \mathbf{S}[j]\Lambda[j] )_+&\text{otherwise}
	\end{cases}
	\end{equation}
	
	Finally, we define
	\begin{equation}
	\tilde{\bQ}[j] = \begin{cases}
	\tilde{\bQ}^{- }[j]&\tilde{T}[2s\T] \leq j<\tilde{T}[(2s+1)\T]\\
	&\text{ for some }s\in\mathbb{N}\\
	\tilde{\bQ}^{+}[j]&\tilde{T}[(2s-1)\T ] \leq j<\tilde{T}[2s\T]\\
	&\text{ for some }s\in\mathbb{N}\\
	\tilde{\bQ}^{-}[j]&j<\tilde{T}[\T]
	\end{cases}
	\end{equation}
	
	Under this coupling, we have the following observation:
	\begin{observation}\label{obs: 1}
		Suppose that $T_j\leq t< T_{j+1}$, then if the random walkers do not visit a vertex $i$ within time $[T_{j-1}, t)$, then $\tilde{Q}_i(t) = Q_i(t)$.
	\end{observation}
	
	Define $\tilde{\bX}(t)$ by
	\begin{equation}
	\tilde{X}_i(t) = \dfrac{1}{n}\sum_{j=1}^n \mathbbm{1}_{\{\tilde{Q}_j(t) \geq i \}}\qquad i\in\mathbb{Z}_+
	\end{equation}
	
	Similarly, define $\tilde{\bX}^+(t),\tilde{\bX}^-(t)$ to be the proportion vectors corresponding to $\tilde{\bQ}^+(t)$ and $\tilde{\bQ}^-(t)$, respectively. 
	
	We further define $\overline{\bX}(t)$ as follows:
	\begin{align*}
	\overline{\bX}(t) = \tilde{\bX}(T_j)\qquad T_j\leq t<T_{j+1}
	\end{align*}
	where $T_j$ is the (continuous) time of the $j$-th arrival. In other words, $\overline{\bX}(t)$ is obtained from sampling and holding $\tilde{\bX}(t)$ at arrival events. Effectively, $\overline{\bX}(t)$ is a process which accumulates services at arrival times.
	
	We write
	\begin{align*}
	\bM^a(t) &= \left[\dfrac{\mathbf{A}(t)}{n}-\int_0^t \mathbf{a}(\overline{\bX}(u))\d u \right] \\
	&\quad + \int_0^t [\mathbf{a}(\overline{\bX}(u)) -\mathbf{a}(\tilde{\bX}(u)) ]\d u \\
	&\quad + \int_0^t [\mathbf{a}(\tilde{\bX}(u)) - \mathbf{a}(\mathbf{X}(u))]\d u \\
	&=:\bM^{a, 1}(t) + \bM^{a, 2}(t) + \bM^{a, 3}(t)
	\end{align*}
	
	Since
	\begin{align*}
	&\quad \|\bX(t) - \tilde{\bX}(t) \|_1\\
	&\leq \dfrac{1}{n}\sum_{l=1}^n |Q_l(t) - \tilde{Q}_l(t)|\\
	&= \dfrac{1}{n}\sum_{l=1}^n (Q_l(t) - \tilde{Q}_l(t))\leq \dfrac{2\T}{n}\qquad\forall t\geq 0
	\end{align*}
	and $\mathbf{a}$ is $2\lambda d$-Lipschitz, we can bound $\|\bM^{a, 3}(t)\|$ by
	\begin{equation}
	\|\bM^{a, 3}(t) \|_1 \leq \int_0^t 2\lambda d \dfrac{2\T}{n}d u = \dfrac{2\lambda d t\T}{n}
	\end{equation}
	
	Hence,
	\begin{equation}
	\sup_{t\in [0, T]}\|\bM^{a, 3}(t)\|_1\leq \dfrac{4\lambda d T\T}{n}
	\end{equation}
	
	Now we try to bound $\|\bM^{a, 2}(t)\|_1$ and $\|\bM^{a, 1}(t)\|_1$. Let $r>1$ be a constant that we choose later. We now introduce two high probability events
	\begin{equation}
	\begin{split}
	\mathcal{A}&:=\{\text{there are strictly less than $e\lambda T n$ job arrivals before $T$} \}\\
	\mathcal{B}&:=\mathcal{A}\cap \{\text{no queue has accepted more than $\kappa_r L\log n$} \\
	&\qquad\text{arrival jobs within the first $\lfloor e\lambda T n\rfloor$ arrivals} \}
	\end{split}
	\end{equation}
	where $L = d\left(\frac{c\log(k-1)}{2\alpha}+1\right)$ and $\kappa_r = \max\{\frac{32r}{3}, 4e\lambda T+1\}$.
	
	\begin{lemma}
		We have
		\begin{equation}
		\Pr(\mathcal{A}^c) \leq e^{-\lambda n T},\quad \Pr(\mathcal{B}^c) \leq e^{-\lambda n T} + 2dn^{-r+1}
		\end{equation}
	\end{lemma}
	
	\begin{proof}
		By Chernoff bound we know that
		\begin{align*}
		\Pr(\mathcal{A}^c) &= \Pr(\mathrm{Poisson}(\lambda n T)\geq e\lambda n T)\\
		&\leq e^{-se\lambda n T}\exp(\lambda n T(e^s-1))\qquad \forall s>0,
		\end{align*}
		and then picking $s=1$ results in
		\begin{align*}
		\Pr(\mathcal{A}^c)\leq e^{-\lambda n T}
		\end{align*}
		
		Now we estimate the probability of $\mathcal{B}^c$: If $\mathcal{B}$ is not true, then either $\mathcal{A}$ is not true, or some queue is allocated with more than $\kappa_r d\log n$ arrival jobs while $\mathcal{A}$ is true. Let $V_{i, l}\langle j\rangle$ denote the number of visits to queue $i$ by the $l$-th random walker between (discrete) time $[(j-1)\T, j\T)$. Let $\tilde{L}=\frac{c\log(k-1)}{2\alpha}+1$. Because of the girth assumption, we have $V_{i, l}\langle j\rangle\leq \tilde{L}$ for all $i, l, j$.
		
		Let $N$ be the smallest even number that is at least $\frac{e\lambda T n}{\T}$. We have
		\begin{align*}
		\Pr(\mathcal{B}^c)&\leq \Pr(\mathcal{A}^c) + \Pr\left(\exists i\in [n], \sum_{l=1}^d\sum_{j=1}^{N} V_{i, l}\langle j\rangle \geq \kappa_r L\log n \right)\\
		&\leq \Pr(\mathcal{A}^c) + n d  \Pr\left(\sum_{j=1}^{N} V_{1, 1}\langle j\rangle \geq \kappa_r \tilde{L} \log n \right)\\
		&\leq \Pr(\mathcal{A}^c) + n d  \Pr\left(\sum_{j=1}^{N} \tilde{L}\mathbbm{1}_{V_{1, 1}\langle j\rangle>0} \geq \kappa_r \tilde{L} \log n \right)\\
		&=\Pr(\mathcal{A}^c) + n d  \Pr\left(\sum_{j=1}^{N} \mathbbm{1}_{V_{1, 1}\langle j\rangle>0} \geq \kappa_r \log n \right)
		\end{align*}
		
		Let $\mathcal{G}_j:=\mathcal{F}_{T_{j\T}}$ for $j\geq 0$ and $\mathcal{G}_{-1}$ be the trivial $\sigma$-algebra. $\{\mathcal{G}_j\}_{j=-1}^\infty$ forms a filtration, and $V_{1,1}[j]$ is adapted to $\{\mathcal{G}_j\}_{j=1}^\infty$.
		
		By Lemma \ref{lem: mixing} and the Union bound we have
		\begin{equation}
		\E[\mathbbm{1}_{V_{1, 1}\langle j\rangle>0}|\mathcal{G}_{j-2}]\leq \T\left(\dfrac{1}{n} + \dfrac{1}{n^2}\right)\leq \dfrac{2\T}{n}=:\tilde{m}
		\end{equation}
		
		Setting $\kappa_r = \max\{\frac{32r}{3}, 4e\lambda T+1\}$, we have 
		\begin{align}
		\kappa_r \log n&\geq 4e\lambda T\log n + \log n \\
		&\geq 4e\lambda T + \dfrac{8\T}{n}\\
		&\text{(for sufficiently large $n$ that does not depend on $T$)}\\
		&= 2\left(\dfrac{e\lambda T n}{\T} + 2\right) \cdot \dfrac{2\T}{n}\\
		&\geq 2 N \tilde{m}
		\end{align}
		
		Hence we can apply Lemma \ref{lem: bern} and obtain
		\begin{align*}
		&\quad \Pr\left(\sum_{j=0}^{\lfloor N\rfloor} \mathbbm{1}_{V_{i, 1}\langle j\rangle>0} \geq \kappa_r \log n \right) \\
		&\leq 2\exp\left(-\dfrac{3\kappa_r \log n}{32} \right)\leq 2n^{-r}
		\end{align*}
		
		Hence 
		\begin{align*}
		\quad\Pr(\mathcal{B}^c)&\leq \Pr(\mathcal{A}^c) + n d  \Pr\left(\sum_{j=0}^{\lfloor N\rfloor} \mathbbm{1}_{V_{1, 1}\langle j\rangle>0} \geq \kappa_r \log n \right) \\
		&\leq e^{-\lambda n T} + 2dn^{-r+1}
		\end{align*}
	\end{proof}
	
	\begin{remark}
		It is possible (using the techniques in \cite{alon2007non}) to show that the arrival jobs accepted by a single queue within the first $\lfloor e\lambda T n\rfloor$ arrivals is bounded by $O(\frac{\log n}{\log\log n})$ with a sufficiently large probability. However, an $O(\log n)$ bound is enough for our purpose.
	\end{remark}
	
	\begin{lemma}
		\begin{equation}
		\begin{split}
		&\quad \Pr\left(\sup_{t\in [0, T]}\|\bM^{a, 2}(t)\|_1 \geq 2\lambda d T \varepsilon_2, \mathcal{A} \right)\\
		&\leq e(1+\lambda) \lambda T n(1+\lambda)^{-n\varepsilon_2}
		\end{split}
		\end{equation}
	\end{lemma}
	
	\begin{proof}
		Same as the proof of Lemma 3 in \cite{sigm}.
	\end{proof}
	
	The rest of the proof is to bound $\|\bM^{a, 1}(t)\|_1$. To achieve this goal, we first sample the continuous-time process $\bM^{a, 1}(t)$ at times $T_{j\T}$, the time of the $j\T$-th job arrival. Denote
	\begin{equation}
	\bM^{a, 1}\langle j\rangle :=\bM^{a, 1}(T_{j\T -})
	\end{equation}
	
	The key lemma for the proof is stated as follows:
	\begin{lemma}\label{lem: key}
		For all $i\geq$ and $j\geq 0$,
		\begin{align*}
		|\E[M_i^{a, 1}\langle j+1\rangle  - M_i^{a, 1}\langle j\rangle | \mathcal{G}_{j-1}]|\leq \dfrac{2d\T^2}{n^{1+\alpha}}
		\end{align*}
	\end{lemma}
	
	\begin{proof}
		For $0\leq s< \T$, define $I_i\{j, s\}$ to be the indicator of the event that $(j\T+s)$-th arrival is allocated to a queue of load $i-1$.
		
		For ease of notation, define $W_l\{j, s\} :=W_l\{j\T+s\} , T_{j, s} := T_{j\T+s}, \tau_{j, s}:=\tau_{j\T+s}$.
		
		Recall that $\tau_j$ is the inter-arrival time between job $j-1$ and job $j$ (where $\tau_1$ is defined to be the arrival time of the first job.) Just as in \cite{sigm}, we have
		\begin{equation}\label{mdecomp}
		\begin{split}
		&\quad M_i^{a, 1}\langle j+1\rangle - M_i^{a, 1}\langle j\rangle \\
		&=\sum_{s=0}^{\T-1}\left(\dfrac{1}{n}I_i\{j, s\} - a_i(\tilde{\bX}(T_{j,s}))\tau_{j,s+1}\right)
		\end{split}
		\end{equation}
		
		Define $\mathcal{D}_{j, s}$ to be the event that (at least) one of $W_1\{j,s\}, W_2\{j,s\},\cdots, W_d\{j,s\}$ that has been visited by (at least) one of the random walkers at some timestamps $(j-1)_+\T\leq t<j\T+s$.

		Conditioned on $\mathcal{D}_{j, s}^c$, by Observation \ref{obs: 1} we have $Q_{W_l\{j, s\}}(T_{j, s}-) = \tilde{Q}_{W_l\{j, s\} }(T_{j,s}) $.
		
		If $j$ is even, then $ \tilde{Q}_{W_l\{j, s\} }(T_{j,s}) =  \tilde{Q}^-_{W_l\{j, s\} }(T_{j,s})$, then we have
		\begin{align*}
		&\quad\E[I_i\{j, s\}~|~\mathcal{G}_{j-1}]\\
		&=\Pr\left(\min_{l=1,\cdots,d}Q_{W_l\{j, s\}}(T_{j,s}-)=i-1 ~\Big|~ \mathcal{G}_{j-1}\right)\\
		&\leq \Pr\left(\min_{l=1,\cdots,d}Q_{W_l\{j, s\}}(T_{j,s}-)=i-1, \mathcal{D}_{j, s}^c ~\Big|~ \mathcal{G}_{j-1} \right) \\&\quad + \Pr(\mathcal{D}_{j, s}~|~\mathcal{G}_{j-1})\\
		&= \Pr\left(\min_{l=1,\cdots,d}\tilde{Q}^-_{W_l\{j, s\}}(T_{j,s})=i-1, \mathcal{D}_{j, s}^c ~\Big|~ \mathcal{G}_{j-1} \right) \\
		&\quad + \Pr(\mathcal{D}_{j, s}~|~\mathcal{G}_{j-1})\\
		&\leq  \Pr\left(\min_{l=1,\cdots,d}\tilde{Q}^-_{W_l\{j, s\}}(T_{j,s})=i-1 ~\Big|~ \mathcal{G}_{j-1} \right) \\
		&\quad + \Pr(\mathcal{D}_{j, s}~|~\mathcal{G}_{j-1})
		\end{align*}
		
		Similarly we have a lower bound
		\begin{align*}
		&\quad \E[I_i\{j, s\}~|~\mathcal{G}_{j-1}]\\
		&=\Pr\left(\min_{l=1,\cdots,d}Q_{W_l\{j, s\}}(T_{j,s}-)=i-1 ~\Big|~ \mathcal{G}_{j-1}\right)\\
		&\geq \Pr\left(\min_{l=1,\cdots,d}Q_{W_l\{j, s\}}(T_{j,s}-)=i-1, \mathcal{D}_{j, s}^c ~\Big|~ \mathcal{G}_{j-1} \right) \\
		&\geq  \Pr\left(\min_{l=1,\cdots,d}\tilde{Q}^-_{W_l\{j, s\}}(T_{j,s})=i-1~\Big|~ \mathcal{G}_{j-1} \right) \\
		&\quad - \Pr(\mathcal{D}_{j, s}|\mathcal{G}_{j-1})
		\end{align*}
		
		Now, we take advantage of the following important observation:
		\begin{observation}
			Let $j\geq 0$ be even. Given $\mathcal{G}_{j-1}$, $\tilde{\bQ}^{-}(T_{j, s})_{0\leq s\leq \T}$ is conditionally independent of $(\mathbf{W}\{j, s\})_{s\geq 0}$.
		\end{observation}
		
		The observation is true for $j=0$ since $(\tilde{\bQ}^-(T_{0,s}))_{0\leq s<\T}$ is a result of applying potential service schedules to $\tilde{\bQ}(T_{0,0})$, and the service schedules are independent of the random walker positions since the random walkers are initialized independently of the initial queue lengths. Recall that $\mathcal{G}_{-1}$ is defined to be the trivial $\sigma$-algebra. Hence the observation is stating that $(\tilde{\bQ}^-(T_{0,s}))_{0\leq s<\T}$ is independent of $(\mathbf{W}\{0, s\})_{s\geq 0}$
		
		For $j=2$, the observation is true since $(\tilde{\bQ}^-(T_{j,s}))_{0\leq s<\T}$ is a result of applying potential service schedules to $\tilde{\bQ}(T_{j-1,0})$, where $\tilde{\bQ}(T_{j-1,0})$ is $\mathcal{G}_{j-1}$-measurable and the service schedules are conditionally independent of the random walker positions given $\mathcal{G}_{j-1}$.
		
		Using the observation, we conclude that $\tilde{\bQ}(T_{j, s})$ is conditionally independent of $\mathbf{W}\{j, s\}$ conditioning on $\mathcal{G}_{j-1}$. 
		
		By Lemma \ref{lem: mixing}, we have
		\begin{align*}
		\left|\Pr\left(W_l\{j, s\} = v~|~\mathcal{G}_{j-1}\right)- \dfrac{1}{n}\right| &\leq \dfrac{1}{n^2}\qquad \\&\forall v=1,\cdots, n,\quad \forall s\geq 0
		\end{align*}
		for $j\geq 1$. The above is also true for $j=0$ since $\mathcal{G}_{j-1}$ is the trivial $\sigma$-algebra and $\Pr(W_l\{j, s\} = v) = \frac{1}{n}$.
		
		Hence
		\begin{footnotesize}
		\begin{align*}
		&\quad \Pr\left(\min_{l=1,\cdots,d}\tilde{Q}^-_{W_l\{j, s\}}(T_{j,s})= i-1 ~\Big|~\tilde{\bQ}(T_{j, s}), \mathcal{G}_{j-1} \right)\\
		&=\sum_{v_1, \cdots, v_d\in [n]} \Pr\left(W_l\{j, s\} = v_l, \forall l=1,\cdots, d~|~\tilde{\bQ}(T_{j, s}),\mathcal{G}_{j-1}\right)\\ &\quad \times\mathbbm{1}_{\{\min_{l=1,\cdots,d} \tilde{Q}^-_{v_l}(T_{j, s})= i-1 \}}\\
		&=\sum_{v_1, \cdots, v_d\in [n]} \Pr\left(W_l\{j, s\} = v_l, \forall l=1,\cdots, d~|~\tilde{\bQ}(T_{j, s}),\mathcal{G}_{j-1}\right)\\
		&\quad\times \left(\prod_{l=1}^{d}\mathbbm{1}_{\{\tilde{Q}^-_{v_l}(T_{j, s})\geq i-1 \}} - \prod_{l=1}^{d}\mathbbm{1}_{\{\tilde{Q}^-_{v_l}(T_{j, s})\geq i \}} \right) \\
		&=\sum_{v_1, \cdots, v_d\in [n]} \prod_{l=1}^d \Pr\left(W_l\{j, s\} = v_l~|~\mathcal{G}_{j-1}\right)\\
		&\quad\times \left(\prod_{l=1}^{d}\mathbbm{1}_{\{\tilde{Q}^-_{v_l}(T_{j, s})\geq i-1 \}} - \prod_{l=1}^{d}\mathbbm{1}_{\{\tilde{Q}^-_{v_l}(T_{j, s})\geq i \}} \right)\\
		&\leq \sum_{\substack{v_l\in [n]\\1\leq l\leq d }} \prod_{l=1}^d \left(\dfrac{1}{n}+ \dfrac{1}{n^2}\right) \left(\prod_{l=1}^{d}\mathbbm{1}_{\{\tilde{Q}^-_{v_l}(T_{j, s})\geq i-1 \}} - \prod_{l=1}^{d}\mathbbm{1}_{\{\tilde{Q}^-_{v_l}(T_{j, s})\geq i \}} \right)\\
		&=\left(\dfrac{1}{n} + \dfrac{1}{n^2}\right)^d \sum_{\substack{v_l\in [n]\\1\leq l\leq d }} \left(\prod_{l=1}^{d}\mathbbm{1}_{\{\tilde{Q}^-_{v_l}(T_{j, s})\geq i-1 \}} - \prod_{l=1}^{d}\mathbbm{1}_{\{\tilde{Q}^-_{v_l}(T_{j, s})\geq i-1 \}} \right) \\
		&=\left(\dfrac{1}{n} + \dfrac{1}{n^2}\right)^d \left(\prod_{l=1}^d \sum_{v=1}^n \mathbbm{1}_{\{\tilde{Q}^-_{v}(T_{j, s})\geq i-1 \}} - \prod_{l=1}^d \sum_{v=1}^n \mathbbm{1}_{\{\tilde{Q}^-_{v}(T_{j, s})\geq i \}}\right) \\
		&=\left(\dfrac{1}{n} + \dfrac{1}{n^2}\right)^d \left(\prod_{l=1}^d (n\tilde{X}_{i-1}^- (T_{j, s})) - \prod_{l=1}^d (n\tilde{X}_{i}^- (T_{j, s}))\right) \\
		&= \left(1+ \dfrac{1}{n}\right)^d \left[(\tilde{X}_{i-1}^- (T_{j, s}))^d - (\tilde{X}_{i}^- (T_{j, s}))^d\right]
		\end{align*}
		\end{footnotesize}
		
		Similarly,
		\begin{align*}
		&\quad\Pr\left(\min_{l=1,\cdots,d}\tilde{Q}^-_{W_l\{j, s\}}(T_{j,s})= i-1 \Big|\tilde{\bQ}(T_{j, s}), \mathcal{G}_{j-1} \right)\\
		&\geq \left(1- \dfrac{1}{n}\right)^d \left[(\tilde{X}_{i-1}^- (T_{j, s}))^d - (\tilde{X}_{i}^- (T_{j, s}))^d \right]
		\end{align*}
		
		For sufficiently large $n$, we have $(1+\frac{1}{n})^d - 1 \leq \frac{3d}{2n}$ and $(1-\frac{1}{n})^d - 1 \geq -\frac{3d}{2n}$, in which case we have
		\begin{align*}
		&\quad (\tilde{X}_{i-1}^- (T_{j, s}))^d -  (\tilde{X}_{i}^- (T_{j, s}))^d -\dfrac{3d}{2n}\\
		&\leq\Pr\left(\min_{l=1,\cdots,d}\tilde{Q}^-_{W_l\{j, s\}}(T_{j,s})= i-1 \Big|\tilde{\bQ}(T_{j, s}), \mathcal{G}_{j-1} \right) \\
		&\leq (\tilde{X}_{i-1}^- (T_{j, s}))^d -  (\tilde{X}_{i}^- (T_{j, s}))^d + \dfrac{3d}{2n}
		\end{align*}
		
		We then conclude that
		\begin{align}
		&\quad\E[I_i\{j, s\}|\mathcal{G}_{j-1}]\\
		&\leq \Pr\left(\min_{1\leq l\leq d}\tilde{Q}^-_{W_l\{j, s\}}(T_{j,s})=i-1 \Big| \mathcal{G}_{j-1} \right) + \Pr(\mathcal{D}_{j, s}|\mathcal{G}_{j-1})\\
		&= \E\left[ \Pr\left(\min_{1\leq l\leq d}\tilde{Q}^-_{W_l\{j, s\}}(T_{j,s})=i-1 \Big|\tilde{\bQ}(T_{j, s}), \mathcal{G}_{j-1} \right)\Big| \mathcal{G}_{j-1}\right] \\&\quad +  \Pr(\mathcal{D}_{j, s}|\mathcal{G}_{j-1})\\
		&\leq \E\left[ (\tilde{X}_{i-1}^- (T_{j, s}))^d -  (\tilde{X}_{i}^- (T_{j, s}))^d~ \Big| \mathcal{G}_{j-1} \right] + \dfrac{3d}{2n} + \Pr(\mathcal{D}_{j, s}|\mathcal{G}_{j-1})\\
		&= \dfrac{1}{\lambda}\E\left[ a_i(\tilde{\bX} (T_{j, s})) \Big| \mathcal{G}_{j-1} \right] + \dfrac{3d}{2n} + \Pr(\mathcal{D}_{j, s}|\mathcal{G}_{j-1})\label{ijsup}
		\end{align}
		and similarly
		\begin{align}
		&\quad \E[I_i\{j, s\}|\mathcal{G}_{j-1}]\\
		&\geq \dfrac{1}{\lambda}\E\left[a_i (\tilde{\bX} (T_{j, s})) | \mathcal{G}_{j-1} \right] - \dfrac{3d}{2n} - \Pr(\mathcal{D}_{j, s}|\mathcal{G}_{j-1})\label{ijslo}
		\end{align}
		
		By construction $(\tilde{\bX}(T_{j,s}))$ is independent of $\tau_{j, s+1}$ conditioned on $\mathcal{G}_{j-1}$. Recall that $\tau_{j, s+1}$ is an inter-arrival time of jobs with $\E[\tau_{j, s+1}] = \dfrac{1}{\lambda n}$. We have
		\begin{equation}\label{aieq}
		\E\left[a_i(\tilde{\bX}(T_{j, s}))\tau_{j, s+1}|\mathcal{G}_{j-1}\right] = \dfrac{1}{\lambda n}\E\left[a_i(\tilde{\bX}(T_{j, s}))|\mathcal{G}_{j-1}\right]
		\end{equation}
		
		Combining \eqref{mdecomp}, \eqref{ijsup}, \eqref{ijslo}, and \eqref{aieq} we have
		\begin{align}
		&\quad |\E[M_i^{a_1}\langle j+1 \rangle - M_i\langle j\rangle | \mathcal{G}_{j-1} ]|\\
		&\leq \sum_{s=0}^{\T-1} \left|\dfrac{1}{n}\E\left[I_i\{j, s\}|\mathcal{G}_{j-1}\right] - \E\left[a_i(\tilde{\bX}(T_{j, s}))\tau_{j, s+1}|\mathcal{G}_{j-1}\right] \right|\\
		&\leq \sum_{s=0}^{\T-1} \left(\dfrac{3d}{2n^2} + \dfrac{1}{n}\Pr(\mathcal{D}_{j, s}|\mathcal{G}_{j-1}) \right) \label{compensator}
		\end{align}
		for $j\geq 0$ even.
		
		With a parallel argument (replacing $\tilde{\bX}^{-}$ by $\tilde{\bX}^+$), \eqref{compensator} is also true for $j\geq 0$ odd. 
		
		\begin{claim}\label{claim: djs}
			For all $j\geq 0$,
			\begin{align*}
			&\quad \Pr(\mathcal{D}_{j, s}|\mathcal{G}_{j-1}) 
			\\&\leq \dfrac{d(\T+s)}{n^\alpha} + \dfrac{d(d-1)(\T+s)(1+n^{-1})}{n}\quad a.s.
			\end{align*}
		\end{claim}
		
		Given Claim \ref{claim: djs}, we have 
		\begin{align*}
		&\quad |\E[M_i^{a_1}\langle j+1 \rangle - M_i\langle j\rangle | \mathcal{G}_{j-1} ]| \\
		&\leq \sum_{s=0}^{\T-1} \left(\dfrac{3d}{2n^2} + \dfrac{d(\T+s)}{n^{1+\alpha}} \right.\\
		&\quad \left.+ \dfrac{d(d-1)(\T+s)(1+n^{-1})}{n^2} \right)\\
		&\leq \dfrac{d\T}{n}\left(\dfrac{3}{2n}+\dfrac{3\T}{2n^\alpha}+\dfrac{3(d-1)\T}{2n} \right.\\
		&\quad \left.+ \dfrac{3(d-1)\T}{2n^2} \right)\\
		&\leq \dfrac{d\T}{n} \cdot \dfrac{2\T}{n^\alpha}\\
		&\quad \text{(For sufficiently large $n$. Notice that $\alpha<1$.)}
		\end{align*}
		
		\begin{proof}[Proof of Claim]
			For $j=0$, by union bound we have 
			\begin{align*}
			\Pr(\mathcal{D}_{0, s}) &\leq \sum_{l=1}^d \sum_{r=0}^s \Pr(W_l\{0, r \} = W_l\{0, s\} ) \\
			&\quad + \sum_{l_1\neq l_2} \sum_{r=0}^s \Pr(W_{l_1}\{0, r \} = W_{l_2}\{0, s\} )\\
			&= d \sum_{r=0}^s \Pr(W_1\{0, r \} = W_1\{0, s\} ) \\
			&\quad + d(d-1) \sum_{r=0}^s \Pr(W_{1}\{0, r \} = W_{2}\{0, s\} )\\
			&\leq \dfrac{ds}{n^\alpha} + \dfrac{d(d-1)s}{n}
			\end{align*}
			
			For $j\geq 1$, by union bound we have
			\begin{align*}
			&\quad \Pr(\mathcal{D}_{j, s}|\mathcal{G}_{j-1})\\
			&\leq \sum_{l=1}^d \sum_{r=-\T}^s \Pr(W_l\{j, r \} = W_l\{j, s\}|\mathcal{G}_{j-1} ) \\
			&\quad + \sum_{l_1\neq l_2} \sum_{r=-\T}^s \Pr(W_{l_1}\{j, r \} = W_{l_2}\{j, s\} |\mathcal{G}_{j-1})\\
			&= d \sum_{r=-\T}^s \Pr(W_1\{j, r \} = W_1\{j, s\} |\mathcal{G}_{j-1}) \\
			&\quad + d(d-1) \sum_{r=-\T}^s \Pr(W_{1}\{j, r \} = W_{2}\{j, s\} |\mathcal{G}_{j-1})\\
			&\leq \dfrac{d(\T+s)}{n^\alpha} + \dfrac{d(d-1)(\T+s)(1+n^{-1})}{n}
			\end{align*}
		\end{proof}
		
	\end{proof}
	
	Lemma \ref{lem: key} states that $\bM^{a,1}\langle j\rangle$ is very close to satisfying the condition of Lemma \ref{lem: azhoef}, except for a small compensator. The next Corollary applies concentration inequalities to bound $|M_i^{a, 1}(t)|$ for each $i$.
	
	\begin{cor}
		Let
		\begin{align*}
		N&:=\min\left\{m\in 2\mathbb{Z}: m\geq \dfrac{e\lambda T n}{\T}  \right\} = \Theta\left(\dfrac{n}{\log n} \right)\\
		\delta &:= \dfrac{2dN\T^2}{n^{1+\alpha}}=O\left(\frac{\log n}{n^\alpha} \right)\\
		K_r&:=\dfrac{\T}{n} + \dfrac{r\log n(\T+1)}{\lambda n}=O\left(\dfrac{(\log n)^2}{n}\right)
		\end{align*}
		
		Then for sufficiently large $n$, 
		\begin{align}
		&\quad\Pr\left(\sup_{t\in [0, T]}|M_i^{a, 1}(t)|\geq \varepsilon_1 + K_r+\delta,\mathcal{A} \right)\\
		&\leq 4\exp\left(-\dfrac{\varepsilon_1^2}{8NK_r^2}\right)+eNn^{-r}
		\end{align}
		holds for any $\varepsilon_1>0$ for all $i\geq 1$.
	\end{cor}
	
	\begin{proof}
		We first provide a lemma on the tail bound for sum of exponential random variables.
		
		\begin{lemma}\label{lem: sumofexp}
			Let $Z$ be the sum of $\T$ independent exponential random variables with mean $\frac{1}{\lambda n}$, then 
			\begin{equation}
			\Pr\left(Z\geq \dfrac{r\log n(\T+1)}{\lambda n}\right)\leq en^{-r}
			\end{equation}
		\end{lemma}
		
		\begin{proof}
			Lemma 12 in \cite{sigm}.
		\end{proof}
		
		Define
		\begin{align*}
		&\quad Z_i^{a, 1}\langle j\rangle := M_i^{a, 1}\langle j+1\rangle - M_i^{a, 1}\langle j\rangle \\
		&=\dfrac{1}{n}\sum_{s=0}^{\T-1} I_i\{j, s\} - \sum_{s=0}^{\T-1} a_i(\tilde{\bX}(T_{j, s})) \tau_{j, s+1}
		\end{align*}
		
		Now, decompose $M_i^{a, 1}\langle j \rangle = \tilde{M}_i^{a, 1}\langle j\rangle + \Delta_i^{a, 1}\langle i\rangle $ where
		\begin{align*}
		&\quad\tilde{M}_i^{a, 1}\langle j\rangle := M_i^{a, 1}\langle 0\rangle + \sum_{l=0}^{j-1} \left(Z_i^{a, 1}\langle l\rangle - \E[Z_i^{a, 1}\langle l\rangle |\mathcal{G}_{l-1}]\right) 
		\end{align*}
		and
		\begin{align*}
		\Delta_i^{a, 1}\langle j\rangle = \sum_{l=0}^{j-1}   \left( \E[Z_i^{a, 1}\langle l\rangle | \mathcal{G}_{l-1}] \right)
		\end{align*}
		
		We immediately have $\E[\tilde{M}_i^{a, 1}\langle j \rangle|\mathcal{G}_{j-2}] = \tilde{M}_i^{a, 1}\langle j-2 \rangle$ and
		\begin{equation}
		|\Delta_i^{a, 1}\langle j\rangle| \leq j \cdot \dfrac{2d\T^2}{n^{1+\alpha}}
		\end{equation}
		
		Thus,
		\begin{align}
		\sup_{0\leq j\leq N}|M_i^{a, 1}\langle j\rangle| &\leq \sup_{0\leq j\leq N}|\tilde{M}_i^{a, 1}\langle j\rangle| + N\cdot \dfrac{2d\T^2}{n^{1+\alpha}}\\
		&= \sup_{0\leq j\leq N}|\tilde{M}_i^{a, 1}\langle j\rangle| + \delta\qquad \forall j\quad a.s.\label{martingaledecomposition}
		\end{align}
		
		Now we provide a bound for the differences of $\tilde{M}_i^{a, 1}[j]$. Notice that
		\begin{equation}
		\tilde{M}_i^{a, 1}\langle j+1\rangle - \tilde{M}_i^{a, 1}\langle j\rangle = Z_i^{a, 1}\langle j\rangle - \E[Z_i^{a, 1}\langle j\rangle | \mathcal{G}_{j-1}]
		\end{equation}
		
		We know that
		\begin{equation}
		- \sum_{s=0}^{\T-1} \tau_{j, s+1} \leq Z_i^{a, 1}\langle j\rangle \leq \dfrac{\T}{n}\qquad a.s.
		\end{equation}
		
		Hence,
		\begin{align}
		\E[Z_i^{a, 1}\langle j\rangle| \mathcal{G}_{j-1}]&\geq \E\left[- \sum_{s=0}^{\T-1} \tau_{j, s+1}\Big|\mathcal{G}_{j-1} \right] \\&= -\dfrac{\T}{\lambda n}\qquad a.s.
		\end{align}
		and
		\begin{equation}
		\E[Z_i^{a, 1}\langle j\rangle| \mathcal{G}_{j-1}]\leq  \dfrac{\T}{n}\qquad a.s.
		\end{equation}
		
		Hence
		\begin{align*}
		&\quad |\tilde{M}_i^{a, 1}\langle j+1\rangle -\tilde{M}_i^{a, 1}\langle j\rangle |\\
		&=|Z_i^{a, 1}\langle j\rangle - \E[Z_i^{a, 1}\langle j\rangle| \mathcal{G}_{j-1}]|\\
		&\leq \dfrac{\T}{n} + \max\left\{\sum_{s=0}^{\T-1} \tau_{j, s+1} ,~ \dfrac{\T}{\lambda n}  \right\}
		\end{align*}
		
		Define $\mathcal{C}$ to be the event that
		\begin{equation}
		\forall 0\leq j< N, \qquad \sum_{s=0}^{\T-1} \tau_{j, s+1} \leq \dfrac{r\log n(\T+1)}{\lambda n}
		\end{equation}
		
		Under $\mathcal{C}$, we have $|\tilde{M}_i^{a, 1}\langle j+1\rangle - \tilde{M}_i^{a, 1}\langle j\rangle|\leq \frac{\T}{n} + \frac{r\log n(\T+1)}{\lambda n} = K_r$. 
		
		\begin{lemma}[Modified Azuma-Hoeffding]\label{lem: azhoef}
			\begin{equation}
			\Pr\left(\max_{0\leq j\leq N}\left|\tilde{M}_i^{a, 1}\langle j\rangle \right|\geq \varepsilon, \mathcal{C} \right) \leq 4\exp\left(-\dfrac{\varepsilon^2}{4NK_r^2}\right)
			\end{equation}
		\end{lemma}
		
		\begin{proof}
			Fix $i$. For the purpose of exposition, set $\tilde{Z}_j:= \tilde{M}_i^{a, 1}\langle j \rangle - \tilde{M}_i^{a, 1}\langle j-1 \rangle$ and 
			
			By the Union Bound, we have
			\begin{align}
			&\quad \Pr\left(\max_{0\leq j\leq N}\left|\sum_{l=1}^j \tilde{Z}_j\right|\geq \varepsilon, \mathcal{C} \right) \\&\leq \Pr\left(\max_{0\leq j\leq N}\left(\left|\sum_{l=1}^{\lceil j/2\rceil} \tilde{Z}_{2l}\right| + \left|\sum_{l=1}^{\lceil j/2\rceil} \tilde{Z}_{2l-1}\right|\right) \geq \varepsilon, \mathcal{C} \right) \\	
			&\leq \Pr\left(\max_{0\leq j\leq \frac{N}{2}}\left|\sum_{l=1}^{j} \tilde{Z}_{2l}\right| + \max_{0\leq j\leq \frac{N}{2}}\left|\sum_{l=1}^{j} \tilde{Z}_{2l-1}\right| \geq \varepsilon, \mathcal{C} \right) \\	
			&\leq \Pr\left(\max_{0\leq j\leq \frac{N}{2}}\left|\sum_{l=1}^j \tilde{Z}_{2l}\right|\geq \dfrac{\varepsilon}{2}, \mathcal{C} \right) \\
			&\quad + \Pr\left(\max_{0\leq j\leq \frac{N}{2}}\left|\sum_{l=1}^j \tilde{Z}_{2l-1}\right|\geq \dfrac{\varepsilon}{2}, \mathcal{C} \right).\label{cineq0}
			\end{align}
			
			We know that $\{\sum_{l=1}^{2j} \tilde{Z}_{2l}\}_{j=0}^N$ is a martingale w.r.t. $\{\mathcal{F}_{2j} \}$. 
			
			Using the decision tree construction used in Lemma 8.2 of \cite{chung2006concentration}, one can construct random variables $\{Y_j\}_{j=0}^N$ such that
			\begin{itemize}
				\item $\{Y_j\}_{j=0}^N$ is a martingale w.r.t. $\{\mathcal{F}_{2j} \}$
				\item $|Y_j-Y_{j-1}|\leq K_r$
				\item $Y_j=\sum_{l=1}^{2j} \tilde{Z}_{2l}$ under event $\mathcal{C}$
			\end{itemize}
		
			Note that Lemma 8.2 of  \cite{chung2006concentration} is stated for finite state space random variables, but it can be easily generalized to our setting, where the distribution of $\tilde{Z}_{2l}$ is driven by discrete events (i.e. arrivals and services.) 
			
			By Azuma-Hoeffding Inequality for Maxima (Eq 3.30 of \cite{habib2013probabilistic})
			\begin{align}
				\Pr\left(\max_{0\leq j\leq \frac{N}{2}} |Y_j-Y_0|\geq \varepsilon \right)\leq 2\exp\left(-\dfrac{\varepsilon^2}{2(\frac{N}{2}K_r^2)} \right)
			\end{align}
			
			we have
			\begin{align}
			&\quad \Pr\left(\max_{0\leq j\leq \frac{N}{2}}\left|\sum_{i=1}^j \tilde{Z}_{2j}\right|\geq \dfrac{\varepsilon}{2}, \mathcal{C} \right)\\
			&= \Pr\left(\max_{0\leq j\leq \frac{N}{2}}\left|Y_j-Y_0\right|\geq \dfrac{\varepsilon}{2}, \mathcal{C} \right) \\
			&\leq 2\exp\left(\dfrac{(\varepsilon/2)^2}{NK_r^2}\right) = 2\exp\left(\dfrac{\varepsilon^2}{4NK_r^2}\right).\label{cineq1}
			\end{align}
			
			For the same reason,
			\begin{equation}\label{cineq2}
			\Pr\left(\max_{0\leq j\leq \frac{N}{2}}\left|\sum_{i=1}^j \tilde{Z}_{2j-1}\right|\geq \dfrac{\varepsilon}{2}, \mathcal{C} \right)\leq  2\exp\left(\dfrac{\varepsilon^2}{4NK_r^2}\right)
			\end{equation}
			
			Combining \eqref{cineq0},\eqref{cineq1}, and \eqref{cineq2} we prove the result.
		\end{proof}

		Now we have
		\begin{equation}
		\Pr\left(\sup_{0\leq j\leq N} |\tilde{M}_i^{a, 1}\langle j\rangle | \geq \varepsilon_1 ,~ \mathcal{C} \right)\leq  4\exp\left(-\dfrac{\varepsilon_1^2}{4NK_r^2} \right)
		\end{equation}
		
		Combining \eqref{martingaledecomposition} we have
		\begin{equation}
		\Pr\left(\sup_{0\leq j\leq N} |M_i^{a, 1}\langle j\rangle| \geq \varepsilon_1 + \delta,~ \mathcal{C} \right)\leq  4\exp\left(-\dfrac{\varepsilon_1^2}{4NK_r^2} \right)
		\end{equation}
		
		When $\mathcal{C}$ is true, we have that, for any $t$ such that $T_{j, 0}\leq t<T_{j+1, 0}$ for some $0\leq j  < N$,
		\begin{equation}
		|M_i^{a, 1}(t) - M_i^{a, 1}\langle j\rangle |\leq \dfrac{\T}{n} + \sum_{s=0}^{\T-1} \tau_{j, s+1} \leq K_r\quad a.s.
		\end{equation}
		
		Hence
		\begin{align}
		&\quad\Pr\left(\sup_{t\in [0, T_{N\T}]} |M_i^{a, 1}(t)|\geq \varepsilon_1+K_r+\delta ,~\mathcal{C}\right) \\
		&\leq  4\exp\left(-\dfrac{\varepsilon_1^2}{4NK_r^2} \right)
		\end{align}
		
		Same as in \cite{sigm}, using Lemma \ref{lem: sumofexp} and the Union bound we can bound
		\begin{equation}
		\Pr(\mathcal{C}^c) \leq eNn^{-r}
		\end{equation}
		
		Finally, under event $\mathcal{A}$, $\overline{T}:=T_{N\T} \geq T$ (since $N\T\geq e\lambda T n$). Thus
		\begin{align*}
		&\quad\Pr\left(\sup_{t\in [0, T]}|M_i^{a, 1}(t)|\geq \varepsilon_1+K_r+\delta, ~\mathcal{A} \right)\\
		&\leq \Pr\left(\sup_{t\in [0, \overline{T}]}|M_i^{a, 1}(t)|\geq \varepsilon_1+K_r+\delta, ~\mathcal{A}\cap\mathcal{C} \right) + \Pr(\mathcal{C}^c)\\
		&\leq 4\exp\left(-\dfrac{\varepsilon_1^2}{4NK_r^2} \right) + eNn^{-r}
		\end{align*}
	\end{proof}
	
	\begin{lemma}
		Let $\rho > 0$ be any constant. Set 
		\begin{align}
		b&:=\lceil (\rho +\kappa_r L)\log n + L\rceil = \Theta(\log n)\\
		\varphi &:= 2N\T \left(\dfrac{\|\bx(0)\|_1+1}{\rho\log n}\right)^d = \Theta\left(\dfrac{n}{(\log n)^d}\right)
		\end{align}
		
		Then
		\begin{equation}
		\begin{split}
		&\quad \Pr\left(\sup_{t\in [0, T]} \sum_{i=b+1}^\infty |M_i^{a, 1}(t)|\geq \dfrac{\varphi}{n} + \lambda T \left(\dfrac{\|\bx(0)\|_1+1}{\rho \log n + L} \right)^d,~\mathcal{B} \right)\\
		&\leq 2\exp\left(-\dfrac{3\varphi}{32\T} \right)
		\end{split}
		\end{equation}
		
	\end{lemma}
	
	\begin{proof}
		By assumption, $\lim_{n\rightarrow\infty}\|\bX^{(n)}(0) - \bx(0)\|_1 = 0$, hence for sufficiently large $n$, we have 
		\begin{align}
		\dfrac{1}{n}\sum_{i=1}^n Q_i^{(n)}(0) = \|\bX^{(n)}(0)\|_1 \leq \|\bx(0)\|_1+1
		\end{align}
		
		Recall that the event $\mathcal{B}$ is an event in which the number of jobs each queue accepted in first $\lfloor e\lambda T n\rfloor$ arrivals is upper bounded by $\kappa_r L\log n$.
		Under event $\mathcal{B}$, the queues with length at least $(\rho+\kappa_r L\log n)$ at any time before the $\lfloor e\lambda T n\rfloor$-th arrival must have an initial length of at least $\rho\log n$. Hence 
		\begin{align}
		&\quad \sup_{0\leq j < e\lambda T n} X_{\lceil (\rho + \kappa_r L) \log n\rceil} \{j\} \leq X_{\lceil \rho\log n\rceil}\{0\}\\
		&\leq \dfrac{\frac{1}{n} \sum_{i=1}^n Q_i^{(n)}(0) }{\rho\log n} \leq \dfrac{\|\bx(0)\|_1+1}{\rho\log n}
		\end{align}
		
		Let $A_{> b}(t):= \sum_{i=b+1}^\infty A_i(t)$ to be the number of arrivals dispatched to queues with length at least $b$ before time $t$ (Recall that $A_i(t)$ is the number of arrivals dispatched to queues with length equal to $i-1$ before time $t$). Let $\tilde{A}_{>b}(t)$ be the number of arrival jobs that satisfy the following conditions:
		\begin{itemize}
			\item The job arrives before time $t$
			\item The assigned queue has length at least $b$ (just before arrival time)
			\item Either this job is the $m$-th job where $m < \T$, or if this job is the $m$-th job where ${j\T}\leq m<{(j+1)\T}$ for some $j\geq 1$ and
			\begin{equation}
			X_{\lceil (\rho+\kappa L)\log n\rceil }(T_{(j-1)\T} -)\leq \dfrac{\|\bx(0) \|_1+1}{\rho\log n}
			\end{equation}
		\end{itemize}
		
		Denote $\tilde{A}_{> b}\langle j\rangle :=\tilde{A}_{> b}(T_{j\T}-)$. Set $b:=\lceil (\rho+\kappa L)\log n + L\rceil$. We have
		\begin{align}
		&\quad\E[\tilde{A}_{> b}\langle j+1\rangle - \tilde{A}_{> b}\langle j\rangle~|~\mathcal{G}_{j-1}] \\
		&\leq \T \left(1+ \dfrac{1}{n}\right)^d \left(\dfrac{\|\bx(0) \|_1+1}{\rho\log n}\right)^d 
		\end{align}
		
		We further have 
		\begin{equation}
		\tilde{A}_{> b}\langle j+1\rangle - \tilde{A}_{> b}\langle j\rangle\leq \T\qquad a.s.
		\end{equation}
		
		Set $\varphi = 2N\T \left(1+ \dfrac{1}{n}\right)^d \left(\dfrac{\|\bx(0) \|_1+1}{\rho\log n}\right)^d$, using Bernstein's Inequality (Lemma \ref{lem: bern}) we obtain
		\begin{equation}
		\Pr\left(\tilde{A}_{> b} \langle N \rangle \geq \varphi \right) \leq 2\exp\left(-\dfrac{3\varphi}{32\T} \right)
		\end{equation}
		
		Under event $\mathcal{B}$, we have $A_{> b}(t) = \tilde{A}_{>b}(t)$ for $t\in [0, T]$, and $T_{N\T} \geq T$ (since $N\T\geq e\lambda T n$). Hence 
		\begin{align}
		&\quad \Pr\left(\sup_{t\in [0, T]} \sum_{i=b+1}^\infty A_i(t) \geq \varphi, ~\mathcal{B} \right)\\
		&=\Pr\left(\sup_{t\in [0, T]} A_{> b}(t) \geq \varphi , ~\mathcal{B}\right)\\
		&\leq \Pr\left(A_{> b}\langle N\rangle \geq \varphi , ~\mathcal{B}\right)\\
		&= \Pr\left(\tilde{A}_{> b}\langle N\rangle \geq \varphi , ~\mathcal{B}\right)\leq 2\exp\left(-\dfrac{3\varphi}{32\T} \right)
		\end{align}
		
		Under event $\mathcal{B}$, we also have
		\begin{equation}
		\overline{X}_{b+1}(t) \leq X_{\lceil \rho\log n + L\rceil }\{0\}\leq \dfrac{\|\bx(0)\|_1+1}{\rho \log n + L}\qquad\forall t\in [0, T]
		\end{equation}
		which implies that
		\begin{align}
		&\quad \sup_{t\in [0, T]} \sum_{i=b+1}^\infty \int_0^t a_i(\overline{\bX}(u))\d u\\
		&\leq \sup_{t\in [0, T]} \int_0^t \lambda (\overline{X}_{b+1}(u))^d \d u\\
		&\leq \lambda T \left(\dfrac{\|\bx(0)\|_1+1}{\rho\log n + L}\right)^d 
		\end{align}
		
		Then, applying union bound we obtain
		\begin{align}
		&\quad\Pr\left( \sup_{t\in [0, T]} \sum_{i=b+1}^\infty |M_i^{a, 1}(t)|\geq \dfrac{\varphi}{n} +\lambda T \left(\dfrac{\|\bx(0)\|_1+1}{\rho\log n + L}\right)^d  \right)\\
		&\leq 2\exp\left(-\dfrac{3\varphi}{32\T} \right)
		\end{align}
	\end{proof}
	
	\begin{cor}
		For any $\varepsilon_1>0$, let $\tilde{\varepsilon}_0=b(\varepsilon_1+K_r+\delta) +\frac{\varphi}{n} + \lambda T \left(\frac{\|\bx(0)\|_1+1}{\rho\log n + L} \right)^d$, we have
		\begin{align}
		&\quad\Pr\left(\sup_{t\in [0, T]} \|\bM^{a, 1}(t)\|_1 \geq \tilde{\varepsilon}_0, ~\mathcal{B} \right)\\
		&\leq b\left[4\exp\left(-\dfrac{\varepsilon_1^2}{4NK_r^2}\right) + eNn^{-3} \right] + 2\exp\left(-\dfrac{3\varphi}{32\T} \right)
		\end{align}
	\end{cor}
	
	\begin{proof}
		Application of the Union bound, similar to proof of Corollary 2 in \cite{sigm}.
	\end{proof}
	
	Now we have provided bounds for $\|\bM^{a, 1}(t)\|_1, \|\bM^{a, 2}(t)\|$ and $\|\bM^{a, 3}(t)\|_1$. Combine all the above, applying the Union bound, we obtain
	\begin{align}
	&\quad \Pr\left(\sup_{0\leq t\leq T} \|\bM^{a}(t)\|_1 \geq \tilde{\varepsilon_0} + 2\lambda d T \varepsilon_2 + \dfrac{4\lambda d T\T }{n},~\mathcal{B} \right)\\
	&\leq b\left[4\exp\left(-\dfrac{\varepsilon_1^2}{4NK_r^2}\right) + eNn^{-r} \right] + 2\exp\left(-\dfrac{3\varphi}{32\T} \right) \\
	&\quad + e(1+\lambda) \lambda T n (1+\lambda)^{-n\varepsilon_2}
	\end{align}
	
	Now we bound $\|\bM^b (t)\|_1$. This part of the proof is nearly identical to that of \cite{sigm}, hence
	\begin{lemma}
		\begin{align}
		&\quad \Pr\left(\sup_{0\leq t\leq T} \|\bM^b(t)\|_1\geq b\varepsilon_3 + \dfrac{(e+1)T(\|\bx(0)\|_1+1)}{\rho\log n+L},~\mathcal{B} \right)\\
		&\leq 2b\exp\left(-nTh\left(\dfrac{\varepsilon_3}{T}\right) \right) + \exp\left(-nT\dfrac{\|\bx(0)\|_1+1}{\rho\log n + L}\right)
		\end{align}
	\end{lemma}
	
	\begin{proof}
		Identical to that of Lemma 6 of \cite{sigm}
	\end{proof}
	
	Now, define
	\begin{align}
	&\quad \varepsilon_0 := b(\varepsilon_1+K_r+\delta) +\frac{\varphi}{n} + \lambda T \left(\frac{\|\bx(0)\|_1+1}{\rho\log n + L} \right)^d + 2\lambda d T \varepsilon_2 \\&+ \dfrac{4\lambda d T\T }{n} + b\varepsilon_3 + \dfrac{(e+1)T(\|\bx(0)\|_1+1)}{\rho\log n+L}
	\end{align}
	
	Combining all the above bounds, we have
	\begin{align}
	&\quad \Pr\left(\sup_{t\in [0, T]} \| \bM(t)\|_1 \geq \varepsilon_0 \right)\\
	&\leq \Pr\left(\sup_{0\leq t\leq T} \|\bM^{a}(t)\|_1 \geq \tilde{\varepsilon_0} + 2\lambda d T \varepsilon_2 + \dfrac{4\lambda d T\T }{n},~\mathcal{B} \right)\\
	&+ \Pr\left(\sup_{0\leq t\leq T} \|\bM^b(t)\|_1\geq b\varepsilon_3 + \dfrac{(e+1)T(\|\bx(0)\|_1+1)}{\rho\log n+L},~\mathcal{B} \right) \\& +\Pr(\mathcal{B})^c\\
	&\leq b\left[4\exp\left(-\dfrac{\varepsilon_1^2}{4NK_r^2}\right) + eNn^{-r} \right] + 2\exp\left(-\dfrac{3\varphi}{32\T} \right) \\
	&\quad + e(1+\lambda) \lambda T n (1+\lambda)^{-n\varepsilon_2} + 2b\exp\left(-nTh\left(\dfrac{\varepsilon_3}{T}\right) \right) \\
	&\quad + \exp\left(-nT\dfrac{\|\bx(0)\|_1+1}{\rho\log n + L}\right) + e^{-\lambda n T} + 2dn^{-r+1}\\
	&=:p_0
	\end{align}
	
	Recall that $N = \Theta(\frac{n}{\log n}), K = \Theta(\frac{(\log n)^2}{n}), \delta = \Theta(\frac{\log n}{n^{\alpha}}), b = \Theta(\log n), \varphi = \Theta(\frac{n}{(\log n)^d})$. 
	
	Select $\varepsilon_1 = \sqrt{4(r-1)NK_r}\log n$, $\varepsilon_2 = \frac{r}{\log(1+\lambda)} \frac{\log n}{n^\alpha}$, $\varepsilon_3 = (r-1)\sqrt{\frac{T\log n}{n}}$, using the fact that $h(t) = \frac{t^2}{2} + o(t^2)$, we finally have
	\begin{equation}
	\varepsilon_0 = o(1),\quad p_0 = O\left(\dfrac{\log n}{n^{r-1}}\right)
	\end{equation}
	
	
	Choose $r=3$. The rest of the proof finishes with Gronwall's lemma and the Borel-Cantelli Lemma in the same way as \cite{sigm}.
	
\end{proof}

\begin{cor}\label{cor: detmain}
	Suppose that
	\begin{enumerate}[(a)]
		\item $(\bQ^{(n)}(0), \overline{\mathbf{W}}(0))$ is deterministic
		\item $\displaystyle\lim_{n\rightarrow\infty}\|\bX^{(n)}(0) - \bx(0)\|_1 = 0$
		\item $\|\bx(0)\|_1<+\infty$
	\end{enumerate} 
	then for every finite $T>0$
	\begin{equation}
	\lim_{n\rightarrow\infty} \sup_{0\leq t\leq T}\|\bX^{(n)}(t) - \bx(t)\|_1 = 0\qquad a.s.
	\end{equation}
\end{cor}

\begin{proof}
	For a moment, assume that $\overline{\mathbf{W}}(0)$ is uniform random on $\vec{\mathcal{E}}^d$ (i.e. random walkers are initialized to independent uniform random positions). Choose $r=d+3$ in the proof of Theorem \ref{thm: main}. Then we obtain
	\begin{align}
	\Pr\left(\sup_{t\in [0, T]} \|\bM(t)\|_1 \geq \varepsilon_0 \right) \leq O\left(\dfrac{\log n}{n^{d+2}}\right)
	\end{align}
	
	For each $\overline{\mathbf{w}}\in \vec{\mathcal{E}}^d$, we have
	\begin{align}
	&\quad \Pr\left(\sup_{t\in [0, T]} \|\bM(t)\|_1 \geq \varepsilon_0 \right) \\&\geq (kn)^{-d} \Pr\left(\sup_{t\in [0, T]} \|\bM(t)\|_1 \geq \varepsilon_0~\Big|~\overline{\mathbf{W}}(0) = \overline{\mathbf{w}} \right) 
	\end{align}
	
	Hence
	\begin{align*}
	\Pr\left(\sup_{t\in [0, T]} \|\bM(t)\|_1 \geq \varepsilon_0~\Big|~\overline{\mathbf{W}}(0) = \overline{\mathbf{w}} \right)\leq  O\left(\dfrac{\log n}{n^{2}}\right)
	\end{align*}
	
	The rest of the proof finishes with Gronwall's lemma and the Borel-Cantelli Lemma.
\end{proof}

\begin{cor}\label{cor: main}
	Let $(\bQ^{(n)}(0), \overline{\mathbf{W}}(0))$ be arbitrarily random and correlated. Suppose that
	\begin{enumerate}[(a)]
		\item $\bX^{(n)}(0)$ converges to $\bY(0)$ weakly in $([0, 1]^{\mathbb{Z}_+}, \|\cdot\|_1 )$
		\item $\|\bY(0)\|_1<+\infty$ a.s.
	\end{enumerate} 
	then for every finite $t>0$, $\bX^{(n)}(t)$ converges to $\bY(t)$ weakly in $([0, 1]^{\mathbb{Z}_+}, \|\cdot\|_1 )$, where $\bY(t)$ is the state of the dynamic system \eqref{diffeq} with random initial state $\bx(0) \stackrel{d}{\sim} \bY(0)$
\end{cor}

\begin{proof}
	$\|\bY(0)\|_1<+\infty$ a.s. means that $\Pr(\bY\in \ell_1([0, 1])) = 1$. Since $\ell_1([0, 1])\subset [0, 1]^{\mathbb{Z}_+}$ is separable with respect to the $\|\cdot\|_1$ metric, we conclude that $\bY$ has separable support. By the Skorokhod representation theorem, there exist a sequence of random vector $(\ddot{\bX}^{(n)}(0))$ and a random vector $\ddot{\bY}(0)$ such that $\ddot{\bX}^{(n)}(0) \stackrel{d}{\sim} \bX^{(n)}(0), \ddot{\bY}(0) \stackrel{d}{\sim} \bY(0)$ and
	\begin{equation}\label{skasconv}
	\lim_{k\rightarrow\infty}\|\ddot{\bX}^{(n)}(0) - 
	\ddot{\bY}(0)\|_1 \xrightarrow{k\rightarrow\infty }0\qquad a.s.
	\end{equation}
	
	Given $(\ddot{\bX}^{(n)}(0))_{n}$, construct the sequence $(\ddot{\bQ}^{(n)}(0), \ddot{\mathbf{W}}^{(n)}(0))_n$ such that $(\ddot{\bQ}^{(n)}(0), \ddot{\mathbf{W}}^{(n)}(0)) \stackrel{d}{\sim} (\bQ^{(n)}(0), \mathbf{W}^{(n)}(0))$, accordingly an on the same probability space.
	
	Construct $(\ddot{\bQ}^{(n)}(t), \ddot{\mathbf{W}}^{(n)}(t))_n$ such that $(\ddot{\bQ}^{(n)}(t), \ddot{\mathbf{W}}^{(n)}(t))$ evolves independently for each $n$. Applying Corollary \ref{cor: detmain} we have 
	\begin{align}
	\Pr\left(\lim_{n\rightarrow\infty} \sup_{0\leq t\leq T}\|\ddot{\bX}^{(n)}(t) - \ddot{\bY}(t)\|_1 = 0~\Big|~(\ddot{\bX}^{(n)}(0) )_n \right) = 1
	\end{align}
	for $\omega$'s such that $\ddot{\bX}^{(n)}(0) (\omega)$ converges in $\ell_1$ to $\ddot{\bY}(0)(\omega)$ and $\|\ddot{\bY}(0)(\omega)\|_1 < +\infty$. 
	
	By \eqref{skasconv}, $\ddot{\bX}^{(n)}(0) (\omega)$ converges in $\ell_1$ to $\ddot{\bY}(0)(\omega)$ and $\|\ddot{\bY}(0)(\omega)\|_1 < +\infty$ for almost all $\omega$, hence we have
	\begin{align}\label{ddotxasconv}
	\Pr\left(\lim_{n\rightarrow\infty} \sup_{0\leq t\leq T}\|\ddot{\bX}^{(n)}(t) - \ddot{\bY}(t)\|_1 = 0 \right) = 1
	\end{align}
	
	In particular, \eqref{ddotxasconv} implies that $\bX^{(n)}(t) \xRightarrow{n\rightarrow\infty} \bY(t)$ for each finite $t$.
\end{proof}

\subsection{Stability of NBRW-Po$d$}\label{sec: stability}
In this section we will show that the proposed scheme stablizes the queuing system for every finite $n$. To achieve this, we need some non-asymptotic assumption on the graph $G^{(n)}$. For this section, we only impose the following minimal assumption on the graph $G$:
\begin{assump}
	$G$ is connected and aperiodic.
\end{assump}

\begin{thm}\label{thm: sta}
	The Markov Process $(\bQ(t), \overline{\mathbf{W}}(t))$ is irreducible and positive recurrent, and hence $\bQ(t)\xRightarrow{t\rightarrow\infty} \hat{\bQ}$ for some $\hat{\bQ}$.
\end{thm}

Recall from \cite{sigm} that we have two down-sampled versions of the process $\bQ(t)$. If sampled at arrivals and potential departures:
\begin{equation}
\bQ[j+1] = (\bQ[j] + \mathbf{R}[j](1-\Lambda[j]) -\mathbf{S}[j]\Lambda[j] )_+
\end{equation}

If sampled at arrivals:
\begin{equation}
\bQ\{j+1\} = (\bQ\{j\} + \mathbf{R}\{j\} - \mathbf{S}\{j\})_+
\end{equation}

\begin{lemma}\label{lem: sub}
	$(\bQ\{j\}, \overline{\mathbf{W}}\{j\})$ is irreducible and positive recurrent.
\end{lemma}

\begin{proof}
	Let $\Omega = \mathbb{Z}_+^n\times \vec{\mathcal{E}}^d$ be the state space of the Markov Chain $(\bQ\{j\}, \overline{\mathbf{W}}\{j\} )$. 
	
	First, we need to show that $(\bQ\{j\}, \overline{\mathbf{W}}\{j\})$ is irreducible: Define $P^{(t)}(\mathbf{q}, \overline{\mathbf{w}}; \mathbf{q}', \overline{\mathbf{w}}'):=\Pr(\bQ\{t\} = \mathbf{q}', \overline{\mathbf{W}}\{t\} = \overline{\mathbf{w}}'~|~\bQ\{0\} = \mathbf{q}, \overline{\mathbf{W}}\{0\} = \overline{\mathbf{w}} )$
	\begin{itemize}
		\item For every $\mathbf{q}\in \mathbb{Z}_+^n, \overline{\mathbf{w}}, \overline{\mathbf{w}}'\in \vec{\mathcal{E}}^d$, $(\mathbf{q}, \overline{\mathbf{w}}')$ is accessible from $(\mathbf{q}, \overline{\mathbf{w}})$ : Since the graph is connected and aperiodic, the non-backtracking random walk on $G$ converges to its stationary distribution: uniform on all directed edges. Hence in particular, there exist a time $K$ such that 
		\begin{equation}
		\Pr(\overline{W}_l\{K\} = \overline{w}_l'~|~\overline{W}_l\{0\} = \overline{w}_l ) > 0\qquad \forall l > 0
		\end{equation}
		which implies that
		\begin{equation}
		\Pr(\overline{\mathbf{W}}\{K\} = \overline{\mathbf{w}}'~|~\overline{\mathbf{W}}\{0\} = \overline{\mathbf{w}} ) > 0
		\end{equation}
		
		With positive probability, $\mathbf{S}\{j\} = \mathbf{R}\{j\}$ for all $j=0,1,\cdots, K-1$ (i.e. in $K$ steps, all assigned jobs ($\mathbf{R}\{j\}$'s) are immediately canceled by services ($\mathbf{S}\{j\}$'s)). Hence 
		\begin{equation}
		P^{(K)}(\mathbf{q}, \overline{\mathbf{w}}; \mathbf{q}, \overline{\mathbf{w}}') > 0
		\end{equation}
		
		\item For every $\mathbf{q}\in \mathbb{Z}_+^n, i\in [n], \overline{\mathbf{w}}\in \vec{\mathcal{E}}^d$, $(\mathbf{q}+\mathbf{e}_i, \overline{\mathbf{w}})$ is accessible from $(\mathbf{q}, \overline{\mathbf{w}})$: Let $\overline{\mathbf{w}}'$ be such that all of $\overline{w}_1', \cdots, \overline{w}_d'$ are pointing towards vertex $i$. There exist $K_1>, K_2>0$ such that
		\begin{equation}
		\Pr(\overline{\mathbf{W}}\{K_1\} = \overline{\mathbf{w}}'~|~\overline{\mathbf{W}}\{0\} = \overline{\mathbf{w}} ) > 0
		\end{equation}
		\begin{equation}
		\Pr(\overline{\mathbf{W}}\{K_1+K_2\} = \overline{\mathbf{w}}~|~\overline{\mathbf{W}}\{K_1\} = \overline{\mathbf{w}}' ) > 0
		\end{equation}
		
		With positive probability, $\mathbf{S}\{j\} = \mathbf{R}\{j\}$ for all $j=0,1,\cdots, K_1-1, K_1+1, \cdots, K_1+K_2-1$ and $\mathbf{S}\{K_1\} = \mathbf{0}$ (i.e. in $K_1+K_2$ steps, all assigned jobs are immediately canceled out by services except the at the $K_1$-th step, where an arrival is assigned to server $i$). We conclude
		\begin{align}
		P^{(K_1+K_2)}(\mathbf{q}, \overline{\mathbf{w}}; \mathbf{q}+\mathbf{e}_i, \overline{\mathbf{w}}) > 0
		\end{align}
		\item For every $i\in [n]$, every $\mathbf{q}\in \mathbb{Z}_+^n$ such that $q_i>0$, $ \overline{\mathbf{w}}\in \vec{\mathcal{E}}^d$, $(\mathbf{q}-\mathbf{e}_i, \overline{\mathbf{w}})$ is accessible from $(\mathbf{q}, \overline{\mathbf{w}})$: Let $K>0$ be such that
		\begin{equation}
		\Pr(\overline{\mathbf{W}}\{K\} = \overline{\mathbf{w}}~|~\overline{\mathbf{W}}\{0\} = \overline{\mathbf{w}} ) > 0
		\end{equation}
		
		With positive probability, $\mathbf{S}\{0\} = \mathbf{R}\{0\} + e_i$, and $\mathbf{S}\{j\}  = \mathbf{R}\{j\}$ for all $j=1,\cdots, K-1$ (i.e. in $K$ steps, all assigned jobs are immediately canceled out by services, and queue $i$ receives one extra service). Hence
		\begin{equation}
		P^{(K)}(\mathbf{q}, \overline{\mathbf{w}}; \mathbf{q}-\mathbf{e}_i, \overline{\mathbf{w}}) > 0
		\end{equation}
		
	\end{itemize}

	From the above argument, we see that every state in $\Omega$ is accessible from every other state. Hence the chain is irreducible.
	
	Define a Lyapunov function $V: \mathbb{Z}_+^n\mapsto \mathbb{R}_+$
	\begin{equation}
	V(\mathbf{q}) := \sum_{i=1}^n q_i^2
	\end{equation}
	
	Now, we compute the drift. For any $j\geq 0$,
	\begin{align}
	&\quad V(\bQ\{j+1\}) - V(\bQ\{j\})\\
	&= \sum_{i=1}^n [(Q_i\{j\} + R_i\{j\} - S_i\{j\}  )_+^2 - (Q_i\{j\})^2]\\
	&\leq \sum_{i=1}^n [(Q_i\{j\} + R_i\{j\} - S_i\{j\}  )^2 - (Q_i\{j\})^2]\\
	&= 2\sum_{i=1}^n Q_i\{j\}(R_i\{j\} - S_i\{j\}) + \sum_{i=1}^n (R_i\{j\} -S_i\{j\})^2\\
	&=-2\sum_{i=1}^n Q_i\{j\} S_i\{j\} + \sum_{i=1}^n (S_i\{j\})^2 + 2Q_{i^*}\{j\} \\
	&\qquad + 1 -2S_{i^*}\{j\}\\
	&\leq -2\sum_{i=1}^n Q_i\{j\} S_i\{j\} + \sum_{i=1}^n (S_i\{j\})^2 + 2Q_{W_1\{j\}}\{j\} \\
	&\qquad + 1 - 2S_{i^*}\{j\} 
	\end{align}
	where $i^*$ is the queue that the $(j+1)$-th job dispatched to, i.e. $\mathbf{R}\{j\} = e_{i^*}$. By the construction of the scheme we know that
	\begin{equation}
	Q_{i^*}\{j\} =\min_{1\leq l\leq d} Q_{W_l\{j\}}\{j\}
	\end{equation}
	
	Since the graph $G$ is assumed to be connected and aperiodic, the non-backtracking random walk on $G$ is an irreducible and aperiodic Markov Chain. Let $K\in \mathbb{N}$ the mixing time of non-backtracking random walk on the $n$ vertex graph in the sense that
	\begin{equation}\label{mixingtimedef}
	\begin{split}
	&\quad \Pr(W_1^{(n)}\{K+1\} =v~|~W_1^{(n)}\{0\} = u_0, W_1^{(n)}\{1\} = u_1) \\
	&\leq \dfrac{1+\lambda}{2\lambda}\cdot \dfrac{1}{n}\qquad\quad \forall u_0, u_1, v\in G
	\end{split}
	\end{equation}
	
	We have
	\begin{align}
	&~\quad \E[V(\bQ\{K + 1\}) - V(\bQ\{K \})~|~\bQ\{0\}, \overline{\mathbf{W}}\{0\} ]\\
	&\leq -2 \E\left[\sum_{i=1}^n Q_i\{K\} S_i\{K\}~\Big|~\bQ\{0\}, \overline{\mathbf{W}}\{0\}\right] \\
	&\quad + \E\left[\sum_{i=1}^n (S_i\{K\})^2~\Big|~\bQ\{0\}, \overline{\mathbf{W}}\{0\} \right] \\
	&\quad + 2\E\left[Q_{W_1\{K\}}\{K\}~|~\bQ\{0\}, \overline{\mathbf{W}}\{0\}  \right] \\
	&\quad + 1 - 2\E[S_{i^*}\{K\}~|~\bQ\{0\}, \overline{\mathbf{W}}\{0\} ] \\
	&= -\dfrac{2}{\lambda n} \E\left[\sum_{i=1}^n Q_i\{K\} ~\Big|~\bQ\{0\}, \overline{\mathbf{W}}\{0\} \right] + n\left[\dfrac{1}{\lambda n} + \dfrac{2}{(\lambda n)^2}\right] \\
	&\quad + 2\E\left[Q_{W_1\{K\}}\{K\}~|~\bQ\{0\}, \overline{\mathbf{W}}\{0\}  \right]  + 1 - \dfrac{2}{\lambda n} \\
	\end{align}
	
	We have
	\begin{align}
	&~\quad \E[Q_{W_1\{K\}}\{K\}~|~\bQ\{0\}, \overline{\mathbf{W}}\{0\} ]\\
	&\leq \E[Q_{W_1\{K\}}\{0\}~|~\bQ\{0\}, \overline{\mathbf{W}}\{0\} ] + K\\
	&\leq \dfrac{1+\lambda}{2\lambda}\cdot\dfrac{1}{n}\sum_{i=1}^n Q_i\{0\} + K
	\end{align}
	
	Using the dynamics of $Q_i\{j\}$ and the fact that $(a)_+\geq a$ we have
	\begin{equation}
	\sum_{i=1}^n Q_i\{K\} \geq \sum_{i=1}^n Q_i\{0\}+ K - \sum_{s=0}^{K-1}\sum_{i=1}^n  S_i\{s\}
	\end{equation}
	
	Hence
	\begin{align}
	&~\quad \E\left[\sum_{i=1}^n Q_i\{K\} ~\Big|~\bQ\{0\}, \overline{\mathbf{W}}\{0\} \right]\\
	&\geq \sum_{i=1}^n Q_i\{0\}+ K - \E\left[\sum_{s=0}^{K-1}\sum_{i=1}^n  S_i\{s\}~\Big|~\bQ\{0\}, \overline{\mathbf{W}}\{0\} \right]\\
	&=\sum_{i=1}^n Q_i\{0\}+ K - Kn\cdot \dfrac{1}{\lambda n} = \sum_{i=1}^n Q_i\{0\} - \left(\dfrac{1}{\lambda} - 1\right)K
	\end{align}
	
	Combining the above we have
	\begin{align}
	&~\quad \E[V(\bQ\{K + 1\}) - V(\bQ\{K \})~|~\bQ\{0\}, \overline{\mathbf{W}}\{0\} ]\\
	&\leq -\dfrac{2}{\lambda n}\left[ \sum_{i=1}^n Q_i\{0\} - \left(\dfrac{1}{\lambda} - 1\right)K \right] + n \left[\dfrac{1}{\lambda n} + \dfrac{2}{(\lambda n)^2}\right]\\
	&\quad + \dfrac{1+\lambda}{\lambda n} \sum_{i=1}^n Q_i\{0\} + 2K + 1 - \dfrac{2}{\lambda n}\\
	&=-\dfrac{1-\lambda}{\lambda n}\sum_{i=1}^n Q_i\{0\} + \dfrac{2}{\lambda n}\left(\dfrac{1}{\lambda}-1\right)K + \dfrac{1}{\lambda} + \dfrac{2}{\lambda^2 n} \\
	&\qquad + 2K + 1 - \dfrac{2}{\lambda n}\\
	&:=-\dfrac{1-\lambda}{\lambda n}\sum_{i=1}^n Q_i\{0\} + C
	\end{align}
	
	Let
	\begin{equation}
	B:= \left\{(\mathbf{q}, \overline{\mathbf{w}})\in\Omega~|~ \sum_{i=1}^n q_i \leq \dfrac{\lambda n (C+1)}{1-\lambda} \right \}
	\end{equation}
	$B$ is a finite subset of $\Omega$.
	
	Then we have
	\begin{align}
	&\quad \E[V(\bQ \{K+1\} )-V(\bQ \{K\} )~|~\bQ\{0\} = \mathbf{q}, \overline{\mathbf{W}}\{0\} = \overline{\mathbf{w}}] \\
	&\leq -1 + (C+1)\mathbbm{1}_{B}(\mathbf{q}, \mathbf{w})
	\end{align}
	
	Furthermore, we have
	\begin{align}
	V(\mathbf{Q}\{j+1\} ) &= \sum_{i=1}^n (Q_i\{j\} + R_i\{j\} - S_i\{j\})_+^2\\
	&\leq \sum_{i=1}^n (Q_i\{j\} + 1)^2 \\
	&= V(\mathbf{Q}\{j\}) + 2\sum_{i=1}^n Q_i\{j\} + n\\
	\end{align}
	
	Hence
	\begin{align}
	&\quad\E[V(\mathbf{Q}\{j+1\} )] \\
	&\leq \E[V(\mathbf{Q}\{j\})] + 2\E\left[\sum_{i=1}^n Q_i\{j\}\right] + n \\
	&\leq \E[V(\mathbf{Q}\{j\})] + 2\sqrt{n\E\left[\sum_{i=1}^n (Q_i\{j\})^2\right]} + n\\
	&\qquad\qquad\qquad  \text{(Cauchy-Schwartz)}\\
	&=(\sqrt{\E[V(\bQ\{j\})]} + \sqrt{n})^2
	\end{align}
	
	Hence whenever $\E[V(\mathbf{Q}\{j\} )] < +\infty$, we have $\E[V(\mathbf{Q} \{j+1\} )] < +\infty$. Therefore the two conditions of Lemma \ref{lem: tsfl} are checked. We conclude that $(\bQ\{j\}, \overline{\mathbf{W}}\{j\} )$ is irreducible and positive recurrent.
	
\end{proof}

\begin{proof}[Proof of Theorem \ref{thm: sta}]
	Using Lemma 8 from \cite{sigm} and the same argument as in \cite{sigm}, we can conclude that $(\bQ[j], \overline{\mathbf{W}}[j])_{j=0}^\infty$ is irreducible and positive recurrent. Then, through standard arguments relating a continuous time Markov Chain (CTMC) to its uniformized chain, the CTMC $(\bQ(t), \overline{\mathbf{W}}(t))$ is positive recurrent. Hence $\bQ(t)\xRightarrow{t\rightarrow\infty} \hat{\bQ}$ for some random vector $\hat{\bQ}$.
	
\end{proof}

\subsection{Convergence of Stationary Distributions}\label{sec: convsta}
In this section, we will show that the stationary distribution of queue lengths converges to the stationary solution of the differential equation as the system size grows. For the results in this section, we impose Assumption \ref{ass: highgirthexpander} on the graph sequence $\{G_n\}_{n}$.

We first provide a refined moment bound.

\begin{lemma}\label{lem: mmb}
	Let $L:=d\left(\frac{c\log(k-1)}{2\alpha}+1\right)$, we have
	\begin{equation}
	\limsup_{n\rightarrow\infty}\E\left[\dfrac{1}{n}\sum_{i=1}^n \hat{Q}_i^{(n)} \right] \leq \dfrac{1+\lambda + 2\lambda L}{1-\lambda}
	\end{equation}
\end{lemma}

\begin{proof}
	Set $K=\T$ where $c>0$ is a constant such that Lemma \ref{lem: mixing} is true, we have
	\begin{equation}
	\begin{split}
	&\quad \Pr(W_1^{(n)}\{K+1\} =v~|~W_1^{(n)}\{0\} = u_0, W_1^{(n)}\{1\} = u_1) \\&\leq \dfrac{1}{n} + \dfrac{1}{n^2}\qquad\quad  \forall u_0, u_1, v\in G^{(n)}
	\end{split}
	\end{equation}
	
	For $n\geq \dfrac{2\lambda}{1-\lambda}$, we have $\dfrac{1}{n} + \dfrac{1}{n^2} \leq \dfrac{1+\lambda}{2\lambda} \dfrac{1}{n}$, hence $K$ satisfies \eqref{mixingtimedef}. From the proof of Lemma \ref{lem: sub} we immediately have
	\begin{align}
	&~\quad \E[V(\bQ\{K + 1\}) - V(\bQ\{K \})~|~\bQ\{0\}, \overline{\mathbf{W}}\{0\} ]\\
	&= -\dfrac{2}{\lambda n} \E\left[\sum_{i=1}^n Q_i\{K\} ~\Big|~\bQ\{0\}, \overline{\mathbf{W}}\{0\} \right] + n\left[\dfrac{1}{\lambda n} + \dfrac{2}{(\lambda n)^2}\right] \\
	&\quad + 2\E\left[Q_{W_1\{K\}}\{K\}~|~\bQ\{0\}, \overline{\mathbf{W}}\{0\}  \right]  + 1 - \dfrac{2}{\lambda n}\label{mb1}
	\end{align}
	and
	\begin{align}
	&~\quad \E\left[\sum_{i=1}^n Q_i\{K\} ~\Big|~\bQ\{0\}, \overline{\mathbf{W}}\{0\} \right]\\
	&\geq \sum_{i=1}^n Q_i\{0\} - \left(\dfrac{1}{\lambda} - 1\right)K\label{mb2}
	\end{align}
	
	To obtain a constant moment bound, we need a tighter upper bound of $\E\left[Q_{W_1\{K\}}\{K\}~|~\bQ\{0\}, \overline{\mathbf{W}}\{0\}  \right]$ than in the proof of Lemma \ref{lem: sub}. Notice that under Assumption \ref{ass: highgirthexpander}, any vertex cannot be visited by the random walkers by more than $L:=d\left(\frac{c\log(k-1)}{2\alpha}+1\right)$ times within $K=\T$ timestamps. Hence
	\begin{align}
	&~\quad \E[Q_{W_1\{K\}}\{K\}~|~\bQ\{0\}, \overline{\mathbf{W}}\{0\} ]\\
	&\leq \E[Q_{W_1\{K\}}\{0\}~|~\bQ\{0\}, \overline{\mathbf{W}}\{0\} ] + L\\
	&\leq \dfrac{1+\lambda}{2\lambda}\cdot\dfrac{1}{n}\sum_{i=1}^n Q_i\{0\} + L\label{mb3}
	\end{align}
	
	Combining \eqref{mb1}\eqref{mb2}\eqref{mb3} we obtain the estimate
	\begin{align}
	&~\quad \E[V(\bQ\{K + 1\}) - V(\bQ\{K \})~|~\bQ\{0\}, \overline{\mathbf{W}}\{0\} ]\\
	&\leq -\dfrac{1-\lambda}{\lambda n}\sum_{i=1}^n Q_i\{0\} + \dfrac{2}{\lambda n}\left(\dfrac{1}{\lambda}-1\right)K + \dfrac{1}{\lambda} + \dfrac{2}{\lambda^2 n} \\
	&\qquad + 2L + 1 - \dfrac{2}{\lambda n}
	\end{align}
	
	Notice that by PASTA, $\hat{\bQ}$, defined as the stationary queue vector of the continuous time process $(\bQ(t), \overline{\mathbf{W}}(t) )$, is also stationary with respect to the Markov Chain $(\bQ\{j\}, \overline{\mathbf{W}}\{j\} )$. Define $f(\mathbf{q}, \overline{\mathbf{w}}) = \dfrac{1-\lambda}{\lambda n}\sum_{i=1}^n q_i$, $g(\mathbf{q}, \overline{\mathbf{w}}) \equiv \dfrac{2}{\lambda n}\left(\dfrac{1}{\lambda}-1\right)K + \dfrac{1}{\lambda} + \dfrac{2}{\lambda^2 n} + 2L + 1 - \dfrac{2}{\lambda n}$, then apply Lemma \ref{lem: momentbound}, we obtain
	\begin{align}
	&\quad \dfrac{1-\lambda}{\lambda}\E\left[\dfrac{1}{n}\sum_{i=1}^n \hat{Q}_i \right] \\
	&\leq \dfrac{2}{\lambda n}\left(\dfrac{1}{\lambda}-1\right)\T + \dfrac{1}{\lambda} + \dfrac{2}{\lambda^2 n} + 2L + 1 - \dfrac{2}{\lambda n}\\&\xrightarrow{n\rightarrow\infty} \dfrac{1+\lambda + 2\lambda L}{\lambda}
	\end{align}
	Multiplying both sides by $\frac{\lambda}{1-\lambda}$ we prove the result.
	
\end{proof}

\begin{lemma}\label{lem: unifint}
	For some $\overline{n}\in\mathbb{N}$, 
	\begin{equation}
	\lim_{b\rightarrow\infty} \sup_{n\geq \overline{n}} \E\left[\sum_{i=b+1}^\infty \hat{X}_i^{(n)} \right] = 0
	\end{equation}
\end{lemma}

\begin{proof}
	This proof is similar to the proof of Lemma 10 in \cite{sigm}.
	
	Define a function $V^b:\mathbb{Z}_+^n \mapsto \mathbb{R}_+$
	\begin{equation}
	V^b(\mathbf{q}) := \sum_{i=1}^n v^b(q_i) := \sum_{i=1}^n (q_i-b)_+^2 
	\end{equation}
	
	Same as \cite{sigm}, we have
	\begin{align}
	&\quad V^b(\bQ\{K+1\}) - V^b(\bQ\{K\}) \\
	&\leq -2\sum_{i=1}^n(Q_i\{K\} - b)_+  S_i\{K\} + \sum_{i=1}^n (S_i\{K\})^2 \mathbbm{1}_{\{ Q_i\{K\} \geq b  \} }\\
	&\quad + 2(Q_{W_1\{K\} } - b)_+ + (1-2S_{i^*}\{K\} ) \mathbbm{1}_{\{ Q_{i^*}\{K\} \geq b  \} }
	\end{align}
	where $i^*$ is the queue that the $(K+1)$-th job is dispatched to.
	
	Set $K = \T$ and let $n\geq \frac{2\lambda}{1-\lambda}$ (such that $K$ satisfies Eq. \eqref{mixingtimedef}). Recall that a server can be visited by the random walkers for at most $L$ times. Hence $\mathbbm{1}_{\{Q_i\{0\} \geq b-L \}} \geq \mathbbm{1}_{ \{ Q_i\{K\} \geq b  \} }$. Therefore
	\begin{align}
	&\quad \E[V^b(\bQ\{K+1\}) - V^b(\bQ\{K\})~|~\bQ\{0\}, \overline{\mathbf{W}}\{0\} ] \\
	&\leq -\dfrac{2}{\lambda n}\E\left[\sum_{i=1}^n(Q_i\{K\} - b)_+~\Big|~\bQ\{0\}, \overline{\mathbf{W}}\{0\} \right]   \\
	&\quad + \left(\dfrac{1}{\lambda n} + \dfrac{2}{\lambda^2 n^2}\right) \sum_{i=1}^n \mathbbm{1}_{\{ Q_i\{0\} \geq b-L  \} } \\
	&\quad + 2\E\left[(Q_{W_1\{K\} } - b)_+~|~\bQ\{0\}, \overline{\mathbf{W}}\{0\}\right] \\
	&\quad + \left(1-\dfrac{2}{\lambda n} \right) \E\left[\mathbbm{1}_{\{ Q_{i^*}\{K\} \geq b  \} }~|~ \bQ\{0\}, \overline{\mathbf{W}}\{0\}\right]
	\end{align}
	
	Same as in \cite{sigm}, we have
	\begin{align}
	&\quad (Q_{W_1\{K\}} \{K\} -b)_+ \\
	&\leq (Q_{W_1\{K\}} \{0\} -b)_+ + L\mathbbm{1}_{ \{Q_{W_1 \{K\} }\{0\} \geq b-L  \} }
	\end{align}
	
	Using the fact that $K$ is a mixing time (i.e. satisfies Eq. \eqref{mixingtimedef}), we have
	\begin{align}
	&\quad \E[(Q_{W_1 \{K\} } \{K\} -b )_+~|~\bQ\{0\}, \overline{\mathbf{W}}\{0\}]\\
	&\leq \E[(Q_{W_1\{K\}} \{0\} -b)_+ ~|~\bQ\{0\}, \overline{\mathbf{W}}\{0\}] \\
	&\quad + L \E[\mathbbm{1}_{ \{Q_{W_1 \{K\} }\{0\} \geq b-L  \} } ~|~\bQ\{0\}, \overline{\mathbf{W}}\{0\} ]\\
	&\leq \dfrac{1+\lambda}{2\lambda} \dfrac{1}{n}\sum_{i=1}^n (Q_i\{0\} - b)_+ + L\cdot \dfrac{1+\lambda}{2\lambda} \dfrac{1}{n} \sum_{i=1}^n \mathbbm{1}_{ \{Q_{i }\{0\} \geq b-L  \} }
	\end{align}
	
	Same as in \cite{sigm}, we have
	\begin{align}
	&\quad \sum_{i=1}^n (Q_i\{K\} - b)_+ \\
	&\geq \sum_{i=1}^n (Q_i\{0\} - b)_+ - \sum_{i=1}^n \sum_{s=0}^{K-1} S_i\{s\} \mathbbm{1}_{ \{Q_{i }\{0\} \geq b-L  \} }
	\end{align}
	
	Thus
	\begin{align}
	&\quad \E\left[\sum_{i=1}^n (Q_i\{K\} - b)_+ ~\Big|~ \bQ\{0\}, \overline{\mathbf{W}}\{0\} \right] \\
	&\geq \sum_{i=1}^n (Q_i\{0\} - b)_+ - \dfrac{K}{\lambda n}\sum_{i=1}^n  \mathbbm{1}_{ \{Q_{i }\{0\} \geq b-L  \} }
	\end{align}
	
	We also have
	\begin{align}
	\mathbbm{1}_{ \{ Q_{i^*} \{j\}\geq b  \} } &\leq \mathbbm{1}_{ \{ Q_{W_1\{j\}} \{j\}\geq b  \} }\leq \mathbbm{1}_{ \{ Q_{W_1\{j\}} \{0\}\geq b-L  \} }
	\end{align}
	
	Hence
	\begin{align}
	&\quad \E[\mathbbm{1}_{ \{ Q_{i^*} \{j\}\geq b  \} }~|~\bQ\{0\}, \overline{\mathbf{W}}\{0\} ] \\
	&\leq \E\left[ \mathbbm{1}_{ \{ Q_{W_1\{j\}} \{0\}\geq b-L  \} }~|~ \bQ\{0\}, \overline{\mathbf{W}}\{0\} \right]\\
	&\leq \dfrac{1+\lambda}{2\lambda} \dfrac{1}{n} \sum_{i=1}^n \mathbbm{1}_{ \{ Q_{i} \{0\}\geq b-L  \} }
	\end{align}
	
	Combining all the above, for sufficiently large $n$ we have
	\begin{align}
	&\quad \E[V^b(\bQ\{K+1\}) - V^b(\bQ\{K\})~|~ \bQ\{0\}, \overline{\mathbf{W}}\{0\} ]\\
	&\leq -\dfrac{2}{\lambda n}\left[ \sum_{i=1}^n (Q_i\{0\} - b)_+ - \dfrac{K}{\lambda n}\sum_{i=1}^n \mathbbm{1}_{ \{Q_{i }\{0\} \geq b-L  \} } \right]   \\
	&\quad + \left(\dfrac{1}{\lambda n} + \dfrac{2}{\lambda^2 n^2}\right) \sum_{i=1}^n \mathbbm{1}_{\{ Q_i\{0\} \geq b-L  \} } \\
	&\quad + 2\left( \dfrac{1+\lambda}{2\lambda} \dfrac{1}{n}\sum_{i=1}^n (Q_i\{0\} - b)_+ \right.\\&\qquad\left.+ L\cdot \dfrac{1+\lambda}{2\lambda} \dfrac{1}{n} \sum_{i=1}^n \mathbbm{1}_{ \{Q_{i }\{0\} \geq b-L  \} } \right) \\
	&\quad + \left(1-\dfrac{2}{\lambda n} \right) \dfrac{1+\lambda}{2\lambda} \dfrac{1}{n} \sum_{i=1}^n \mathbbm{1}_{ \{ Q_{i} \{0\}\geq b-L  \} }\\
	&=-\dfrac{1-\lambda}{\lambda} \dfrac{1}{n} \sum_{i=1}^n (Q_i\{0\} - b)_+ \\
	&\quad + \left(\dfrac{2K}{\lambda^2 n} + \dfrac{1}{\lambda} + \dfrac{2}{\lambda^2 n} + \dfrac{(1+\lambda) L}{\lambda} + \dfrac{1+\lambda}{2\lambda} - \dfrac{1+\lambda}{\lambda^2 n}\right)\\
	&\quad\times\dfrac{1}{n} \sum_{i=1}^n \mathbbm{1}_{ \{ Q_{i} \{0\}\geq b-L  \} }
	\end{align}
	
	Since $(\mathbf{Q}\{j\}, \overline{\mathbf{W}}\{j\} )$ is positive recurrent with stationary queue vector $\hat{\bQ}$, applying Lemma \ref{lem: momentbound} we obtain
	\begin{align}
	&\quad~ \dfrac{1-\lambda}{\lambda} \E\left[\dfrac{1}{n} \sum_{i=1}^n (\hat{Q}_i\{0\} - b)_+ \right] \\
	&\leq \left(\dfrac{2\T}{\lambda^2 n} + \dfrac{1}{\lambda} + \dfrac{2}{\lambda^2 n} + \dfrac{(1+\lambda) L}{\lambda} + \dfrac{1+\lambda}{2\lambda} - \dfrac{1+\lambda}{\lambda^2 n}\right) \\
	&\quad\times \E\left[\dfrac{1}{n}  \sum_{i=1}^n \mathbbm{1}_{ \{ \hat{Q}_{i} \{0\}\geq b-L  \} }\right]
	\end{align}
	
	which means that there exist a constant $\kappa_5$ such that for sufficiently large $n$,
	\begin{align}
	&\quad \E\left[\dfrac{1}{n} \sum_{i=1}^n (\hat{Q}_i\{0\} - b)_+ \right] \\
	&\leq \kappa_5 \E\left[\dfrac{1}{n} \sum_{i=1}^n \mathbbm{1}_{ \{ \hat{Q}_{i} \{0\}\geq b-L  \} }\right] \\
	&\leq \kappa_5 \E\left[\dfrac{1}{n} \sum_{i=1}^n \dfrac{\hat{Q}_i}{b-L}\right] = \dfrac{\kappa_5}{b-L} \E\left[\dfrac{1}{n}\sum_{i=1}^n \hat{Q}_i \right]\qquad \forall b>L
	\end{align}
	
	By Lemma \ref{lem: mmb} we know that there exist $\overline{n}_1\in\mathbb{N}$ such that $\E\left[\dfrac{1}{n}\sum_{i=1}^n \hat{Q}_i \right] \leq \dfrac{1+\lambda + 2\lambda L}{1-\lambda}+1$ for all $n\geq\overline{n}_1$.
	
	Hence, there exist $\overline{n}$ (which does not depend on $b$), such that
	\begin{align}
	\sup_{n\geq \overline{n}} \E\left[\dfrac{1}{n} \sum_{i=1}^n (\hat{Q}_i^{(n)} - b)_+ \right] \leq \dfrac{\kappa_5}{b-L}\left( \dfrac{1+\lambda + 2\lambda L}{1-\lambda}+1\right)
	\end{align}
	
	Observe that
	\begin{align}
	\sum_{i=b+1}^\infty \hat{X}_i^{(n)} = \dfrac{1}{n} \sum_{i=1}^n (\hat{Q}_i^{(n)} - b)_+ 
	\end{align}
	
	We conclude
	\begin{align}
	\lim_{b\rightarrow\infty} \sup_{n\geq \overline{n}} \E\left[ \sum_{i=b+1}^\infty \hat{X}_i^{(n)} \right] = 0
	\end{align}
	
\end{proof}

The rest of the proof follows from similar argument as the proof of Theorem 9 and Theorem 10 in \cite{ying2015power}. For completeness, we include a proof here.

\begin{lemma}\label{lem: tight}
	Every subsequence of $(\hat{\bX}^{(n)})$ has a further subsequence that converges in $d_{W_1, \|\cdot\|_1}$
\end{lemma}

\begin{proof}
	Nearly the same as proof of Corollary 5 in \cite{sigm} (except that, the proof here is simpler).
	
	Endow $[0, 1]^{\mathbb{Z}_+}$ with the metric
	\begin{align}
	\rho(\bx, \mathbf{y}) = \sup_{i\geq 0} \dfrac{|x_i-y_i|}{i+1}
	\end{align}
	
	Under the metric $\rho$, $[0, 1]^{\mathbb{Z}_+}$ is a compact separable metric space. Any collection of probability measures is trivially tight in $([0, 1]^{\mathbb{Z}_+}, \rho)$. By the Prokhorov Theorem, every subsequence of $(\hat{\bX}^(n))$ has a further subsequence that converges weakly in $([0, 1]^{\mathbb{Z}_+}, \rho)$. By the Skorokhod representation theorem, there exist a sequence of random vectors $(\ddot{\bX}^{(n_k)})$ such that $\ddot{\bX}^{(n_k)} \stackrel{d}{\sim} \hat{\bX}^{(n_k)}$ with $(\ddot{\bX}^{(n_k)})$ converges in $([0, 1]^{\mathbb{Z}_+}, \rho)$ almost surely. Denote its almost sure limit by $\bY$.
	
	\begin{claim}\label{claim: expfin}
		$ \sum_{i=1}^\infty\E\left[Y_i \right] < +\infty$
	\end{claim}
	
	\begin{proof}[Proof of Claim]
		By Fatou's Lemma, we have
		\begin{equation}\label{eq1000}
		\sum_{i=1}^\infty Y_i = \sum_{i=1}^\infty \liminf_{k\rightarrow\infty} \ddot{X}_i^{(n_k)}  \leq \liminf_{k\rightarrow\infty} \sum_{i=1}^\infty \ddot{X}_i^{(n_k)}  \qquad a.s.
		\end{equation}
		
		Again, by Fatou's Lemma, we have
		\begin{equation}\label{eq1001}
		\E\left[ \liminf_{k\rightarrow\infty} \sum_{i=1}^\infty \ddot{X}_i^{(n_k)}  \right] \leq \liminf_{k\rightarrow\infty} \E\left[\sum_{i=1}^\infty \ddot{X}_i^{(n_k)} \right]
		\end{equation}
		
		By Lemma \ref{lem: mmb} we know that
		\begin{equation}\label{eq1002}
		\begin{split}
		\liminf_{k\rightarrow\infty}\E\left[\sum_{i=1}^\infty \ddot{X}_i^{(n_k)}\right] &=\liminf_{k\rightarrow\infty}\E\left[\dfrac{1}{n_k}\sum_{i=1}^{n_k} \hat{Q}_i^{(n_k)}\right] \\&\leq \dfrac{1+\lambda + 2\lambda L}{1-\lambda}
		\end{split}
		\end{equation}
		
		Combining \eqref{eq1000}\eqref{eq1001}\eqref{eq1002} we have
		\begin{equation}
		\E\left[\sum_{i=1}^\infty Y_i\right] \leq \dfrac{1+\lambda + 2\lambda L}{1-\lambda }
		\end{equation}
		
	\end{proof}
	
	For any $\varepsilon>0$, by Lemma \ref{lem: unifint} and Claim \ref{claim: expfin}, there exist $b\in\mathbb{Z}_+$ such that
	\begin{equation}
	\sup_{k:~n_k\geq \overline{n}} \E\left[\sum_{i=b+1}^\infty \ddot{X}_i^{(n_k)} \right] \leq \dfrac{\varepsilon}{3},\quad \E\left[\sum_{i=b+1}^\infty Y_i\right]\leq \dfrac{\varepsilon}{3}
	\end{equation}
	which implies
	\begin{equation}
	\E\left[\sum_{i=b+1}^\infty |\ddot{X}_i^{(n_k)} - Y_i| \right] \leq \dfrac{2\varepsilon}{3}
	\end{equation}
	for all $k$ such that $n_k\geq \overline{n}$.
	
	Since $\ddot{X}_i^{(n_k)} - Y_i$ converges to $0$ a.s., and $|\ddot{X}_i^{(n_k)} - Y_i|\leq 2$ a.s., by the Bounded Convergence Theorem we have
	\begin{equation}
	\lim_{k\rightarrow\infty} \E\left[\sum_{i=0}^b |\ddot{X}_i^{(n_k)} - Y_i| \right] = 0
	\end{equation}
	
	Thus, for sufficiently large $k$, $\E\left[\displaystyle\sum_{i=0}^b |\ddot{X}_i^{(n_k)} - Y_i| \right] \leq \dfrac{\varepsilon}{3}$ and
	\begin{equation}
	\begin{split}
	&\quad \E[\|\ddot{\bX}^{(n_k)} - \mathbf{Y}\|_1] \\&=  \E\left[\sum_{i=0}^b |\ddot{X}_i^{(n_k)} - Y_i| \right] + \E\left[\sum_{i=b+1}^\infty |\ddot{X}_i^{(n_k)} - Y_i| \right] \leq \varepsilon
	\end{split}
	\end{equation}
	
	Hence, we established that 
	\begin{equation}
	\lim_{k\rightarrow\infty}\E\left[\|\ddot{\bX}^{(n_k)} -  \mathbf{Y}\|_1\right] = 0
	\end{equation}
	
	which implies that $\hat{\bX}^{(n_k)}$ converges to $\bY$ in $d_{W_1, \|\cdot\|_1}$
	
\end{proof}

\begin{thm}
	$\displaystyle\lim_{n\rightarrow\infty} \E[\|\hat{\bX}^{(n)} - \hat{\bx} \|_1] = 0$, where $\hat{\bx}$ is the unique fixed point of \eqref{diffeq}.
\end{thm}

\begin{proof}
	Let $(\hat{\bQ}^{(n)}, \hat{\mathbf{W}}^{(n)})$ be a random vector whose distribution is the stationary distribution with respect to the $n$-server system. Set $(\bQ^{(n)}(0), \overline{\mathbf{W}}^{(n)}(0))\stackrel{d}{\sim}(\hat{\bQ}^{(n)}, \hat{\mathbf{W}}^{(n)})$ for every $n$. We have $(\bQ^{(n)}(t), \overline{\mathbf{W}}^{(n)}(t))\stackrel{d}{\sim}(\hat{\bQ}^{(n)}, \hat{\mathbf{W}}^{(n)})$ for every $n$ and all $t \geq 0$. In particular, 
	\begin{equation}
	\bX^{(n)}(t) \stackrel{d}{\sim} \hat{\bX}^{(n)}
	\end{equation}
	for all $n$ and $t\geq 0$.
	
	Let $(\hat{\bX}^{(n_k)})_{k=1}^\infty$ be a subsequence of $(\hat{\bX}^{(n)})_n$ that converges in $d_{W_1, \|\cdot\|_1}$. Denote its limit by $\bY$. 
	
	By the proof of Lemma \ref{lem: tight} we know that $\E[\|\bY\|_1] < +\infty$, hence $\|\bY\|_1<+\infty$ a.s. By Corollary \ref{cor: main}, 
	\begin{equation}
	\bX^{(n_k)}(t)\xRightarrow[]{k\rightarrow\infty}\bY(t) \qquad\forall t\in\mathbb{R}_+
	\end{equation}
	where $\bY(t)$ is the state of the limiting dynamical system (specified by the differential equations in \eqref{diffeq}) with initial state given by the random variable $\bx(0)\stackrel{d}{\sim}\bY$. 
	
	Since the weak limit of sequences of random vectors is unique, we obtain 
	\begin{equation}\label{samedist}
	\bY(t)\stackrel{d}{\sim}\bY\qquad\forall t\geq 0
	\end{equation}
	
	It is proved in \cite{mitzenmacher2001power} that the limiting dynamics in \eqref{diffeq} is global asymptotically stable. For any $\bx(0) \in \ell_1([0, 1])$, we have $\bx(t)\xrightarrow{t\rightarrow\infty} \hat{\bx}$. Hence, we have
	\begin{equation}\label{distconv}
	\lim_{t\rightarrow\infty} \bY(t) = \hat{\bx}\qquad a.s.
	\end{equation}
	
	Combining \eqref{samedist} and \eqref{distconv}, we have $\bY = \hat{\bx}$ a.s. We have proved that every subsequence of $(\hat{\bX}^{(n)})_n$ which converges in $d_{W_1, \|\cdot\|_1}$ converges \emph{weakly} to $\hat{\bx}$.
	
	Since the weak limit of a sequence of random vectors is unique, we conclude that every subsequence of $(\hat{\bX}^{(n)})_n$ which converges in $d_{W_1, \|\cdot\|_1}$ converges to $\hat{\bx}$. Since $\hat{\bx}$ is deterministic, we further have this subsequence to converge to $\hat{\bx}$ in $L_1$.
	
	Combining Lemma \ref{lem: tight}, we conclude that every subsequence of $(\hat{\bX}^{(n)})_n$ has a further subsequence that converges to $\hat{\bx}$ in $L_1$. 
	
	Therefore we conclude that $\hat{\bX}^{(n)}$ converges to $\hat{\bx}$ in $L_1$.
\end{proof}

\section{Simulation Results}\label{sec: sim}
While the main results in the paper suggests that the queuing system dynamics converges to the solution of an ODE as the system size goes to infinity, it is not clear how well the ODE approximates the dynamics of the system for any finite number of servers. In this section, we provide simulation results to show that this approximation is still accurate for systems of tens of thousands of servers, which is the same scale as today's cloud computing centers. We also provide results to show that, for some graphs that are low-girth or non-expander, the ODE can fail to capture the dynamics of the system.

We use LPS graph \cite{lubotzky1988ramanujan} as the underlying graph. Specifically, we use a $6$-regular LPS graph of $n=12180$ vertices. The queue length statistics of single sample paths are shown in Figure \ref{fig: simresult1} and \ref{fig: simresult2}. The results show that the fluid-limit approximation can be accurate for relatively small systems.

We also test the scheme on small girth graphs. Specifically, we choose the $6$-regular torus graph $\mathbb{Z}_{15}\times\mathbb{Z}_{28}\times \mathbb{Z}_{29}$ as the underlying graph. The evolution of the queue length statistics are shown in Figure \ref{fig: simresult3} and \ref{fig: simresult4}. The results show that when Assumption \ref{ass: highgirthexpander} is violated, the ODE system may fail to capture the system dynamics.

We also test the scheme on cycle graphs, which are high girth non-expander graphs. The evolution of the queue length statistics are shown in Figure \ref{fig: simresult5} and \ref{fig: simresult6}. The results again show that when Assumption \ref{ass: highgirthexpander} is violated, the ODE system may fail to capture the system dynamics. Note that in \cite{sigm}, with the presence of resets, a similar scheme still yields the same performance as \emph{power-of-$d$-choices}.
\begin{figure}[H]
	\centering
	\includegraphics[width=9cm]{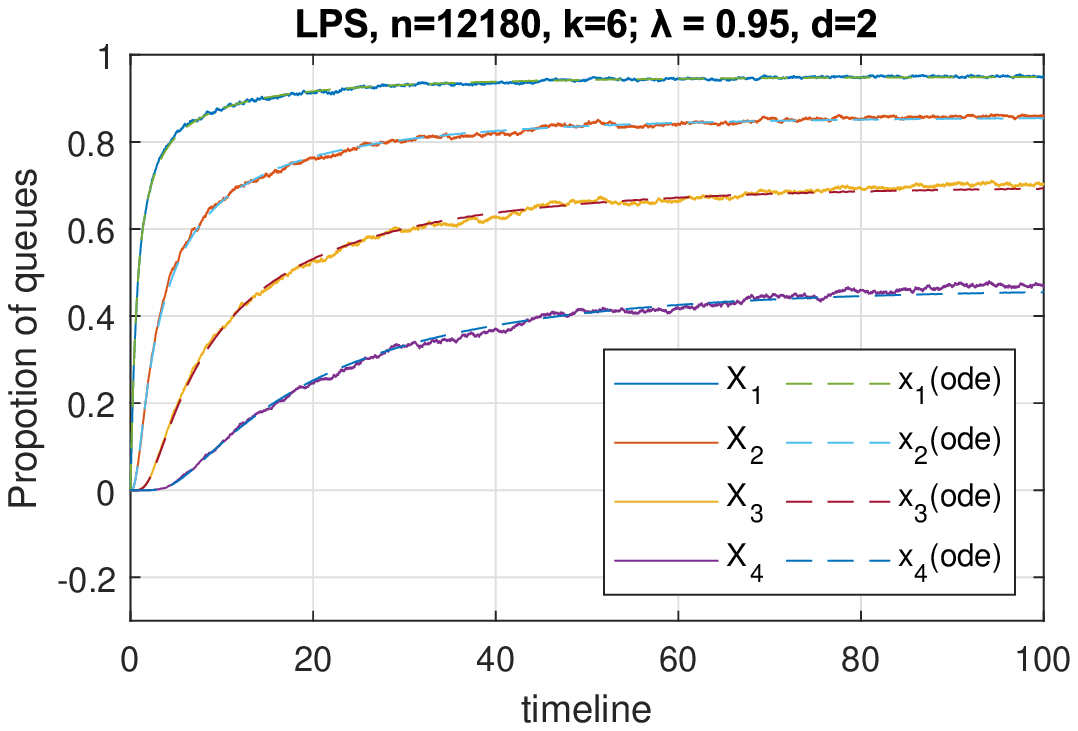}
	\caption{Queue length statistics evolution for NBRW-Po$d$ algorithm with $d=2$ and $\lambda = 0.95$. System is empty at time $0$.}\label{fig: simresult1}
\end{figure}

\begin{figure}[H]
	\centering
	\includegraphics[width=9cm]{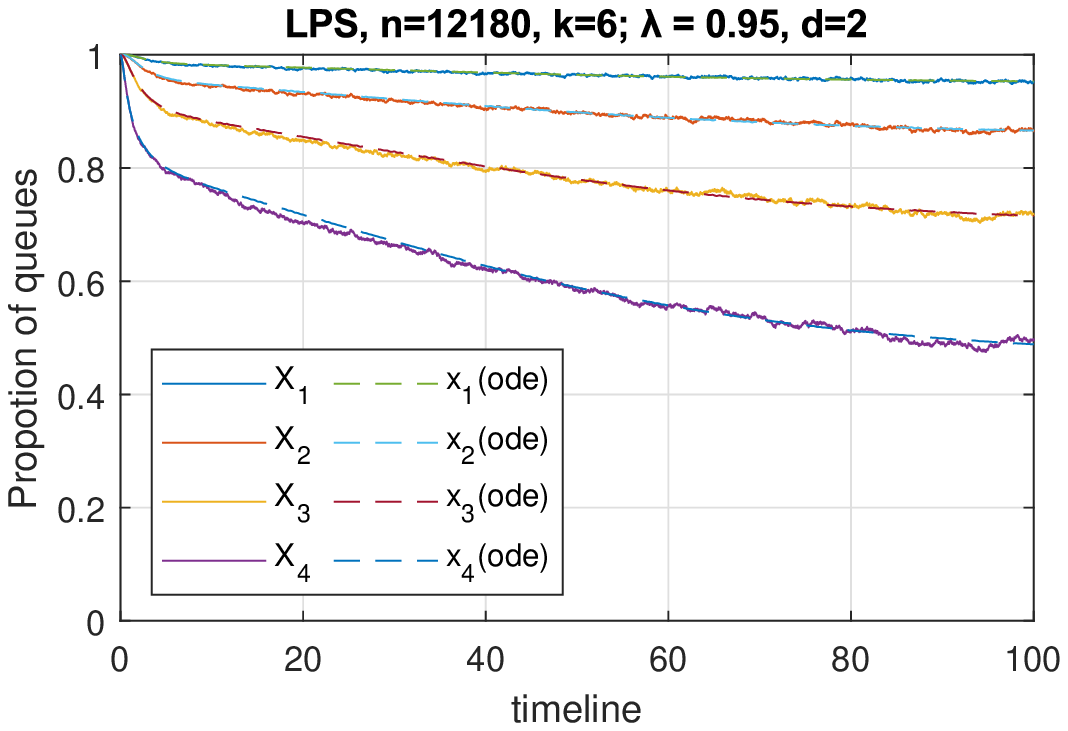}
	\caption{Queue length statistics evolution for NBRW-Po$d$ algorithm with $d=2$ and $\lambda = 0.95$. Each queue has a length of $5$ at time $0$.}\label{fig: simresult2}
\end{figure}

\begin{figure}[H]
	\centering
	\includegraphics[width=9cm]{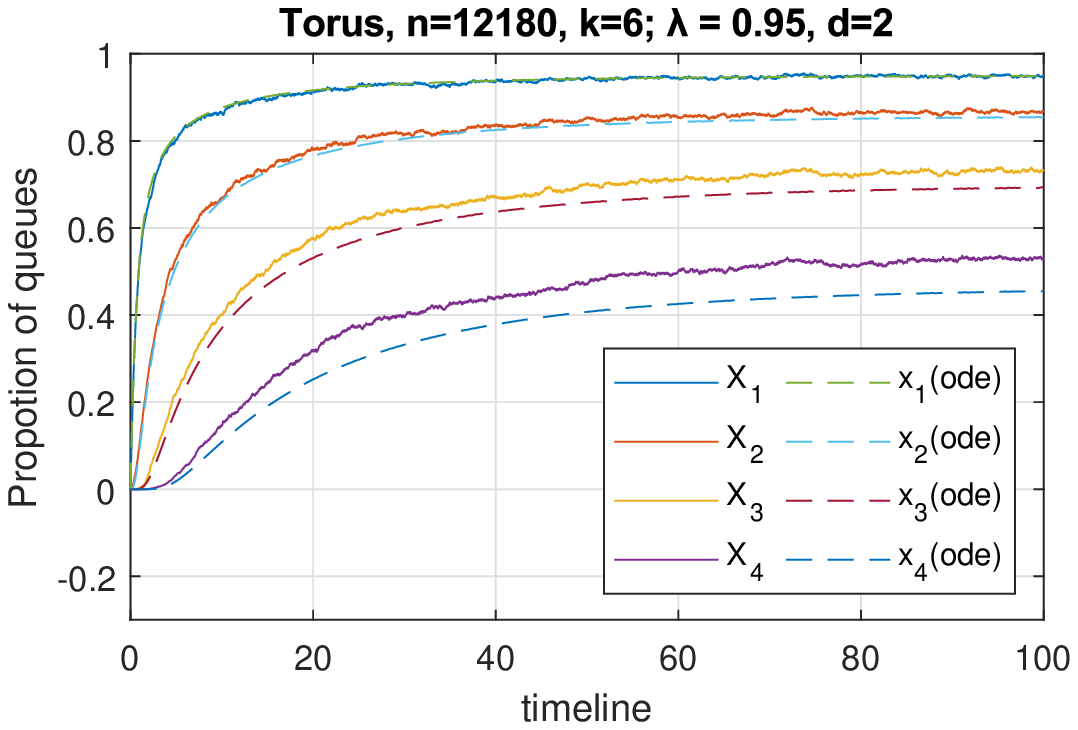}
	\caption{Queue length statistics evolution for NBRW-Po$d$ algorithm with $d=2$ and $\lambda = 0.95$. System is empty at time $0$.}\label{fig: simresult3}
\end{figure}

\begin{figure}[H]
	\centering
	\includegraphics[width=9cm]{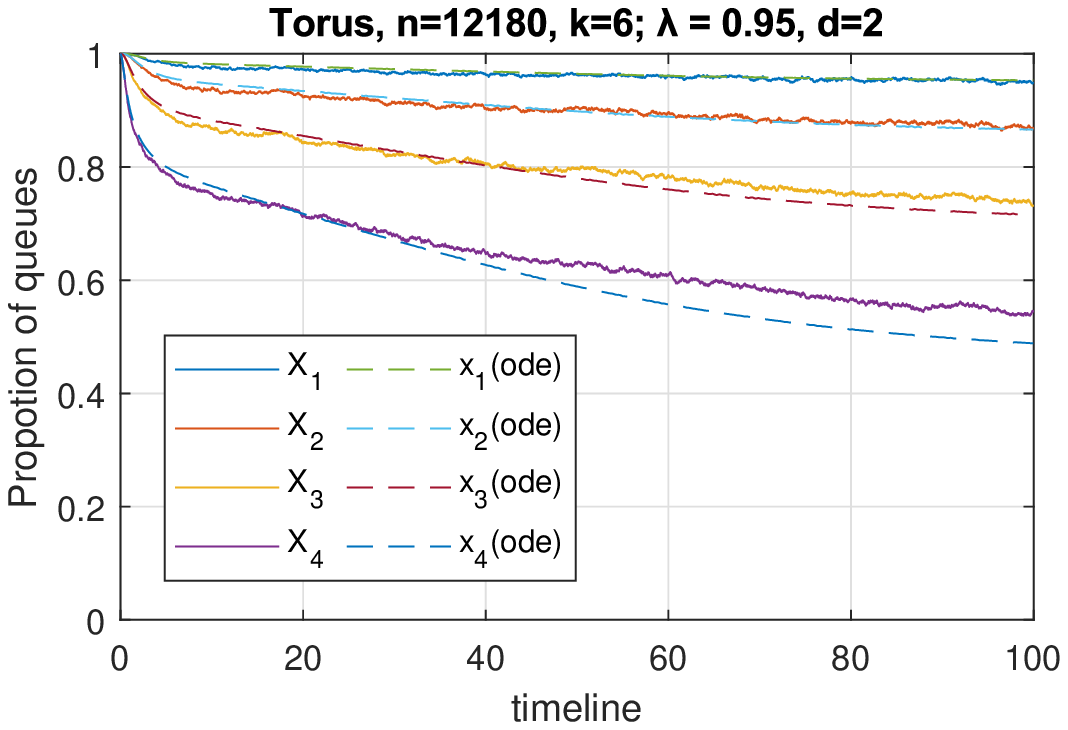}
	\caption{Queue length statistics evolution for NBRW-Po$d$ algorithm with $d=2$ and $\lambda = 0.95$. Each queue has a length of $5$ at time $0$.}\label{fig: simresult4}
\end{figure}

\begin{figure}[H]
	\centering
	\includegraphics[width=9cm]{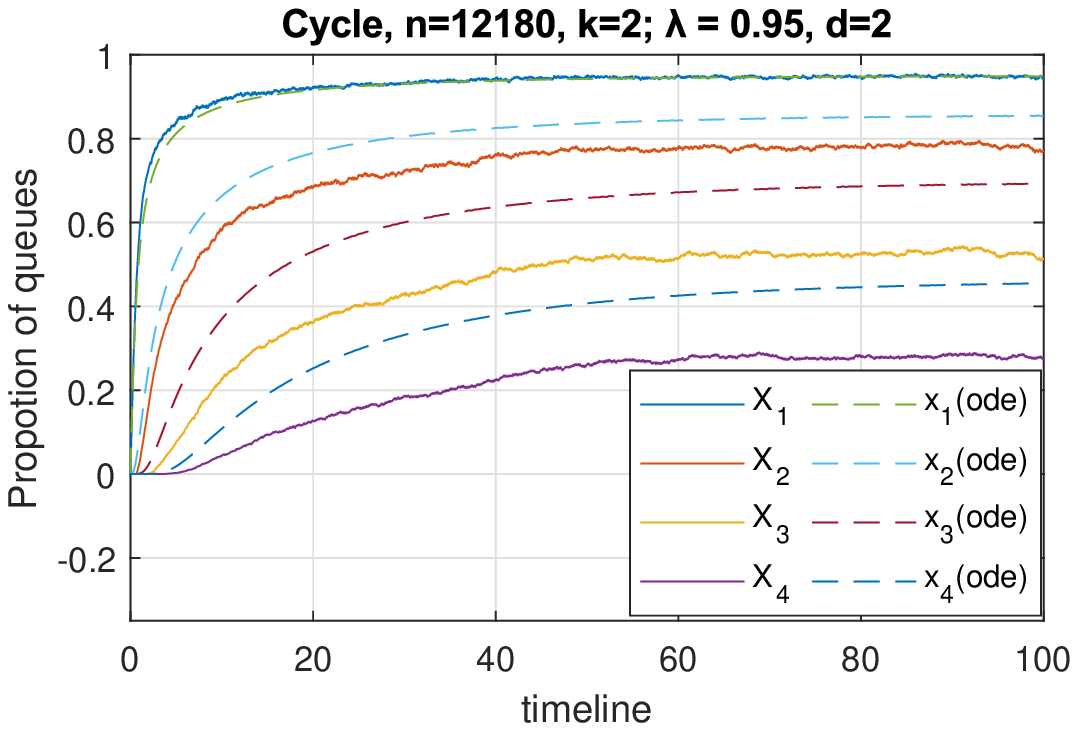}
	\caption{Queue length statistics evolution for NBRW-Po$d$ algorithm with $d=2$ and $\lambda = 0.95$. System is empty at time $0$.}\label{fig: simresult5}
\end{figure}

\begin{figure}[H]
	\centering
	\includegraphics[width=9cm]{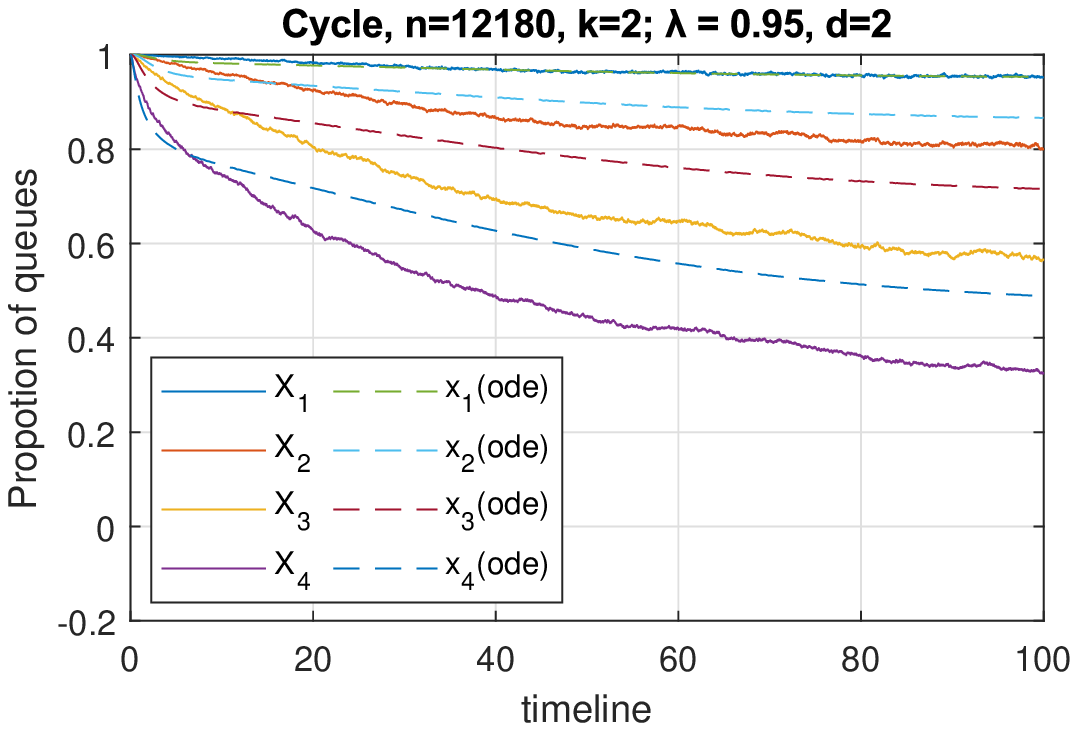}
	\caption{Queue length statistics evolution for NBRW-Po$d$ algorithm with $d=2$ and $\lambda = 0.95$. Each queue has a length of $5$ at time $0$.}\label{fig: simresult6}
\end{figure}

\section{Conclusions}\label{sec: concl}
In this paper we proposed and analyzed a low-randomness load balancing scheme for multi-server systems. The new scheme modifies the sampling procedure of the classical \emph{power-of-$d$-choices} by replacing independent uniform sampling with non-backtracking random walks on high-girth expander graphs. We show that, like \emph{power-of-$d$-choices}, the system dynamics under the new scheme can be approximated by the solution to a system of ODE. We also show that the scheme stablizes the system under mild assumptions. Finally, we show that the stationary queue length distribution of the system under the proposed scheme is the same as that of \emph{power-of-$d$-choices}. We conclude that the new scheme is a derandomization of \emph{power-of-$d$-choices} as it achieves the same performance by using less randomness.

There are a few future research directions suggested by this paper. First, the performance of NBRW-Po$d$ scheme under a heavy traffic model is of interest. Secondly, as the high-girth expander assumption can be too strong, it is worth identifying weaker assumptions in which the results in this paper still holds. Finally, analyzing the structure of the limiting stationary distribution of queue lengths for NBRW-Po$d$ scheme, in particular, if \emph{propagation of chaos} occurs or not , is of interest.

\section*{Acknowledgment}
The authors would like acknowledge support by NSF via grants AST-1343381, AST-1516075, IIS-1538827 and ECCS1608361, and also Harsha Honappa and Yang Xiao for early discussions.


\bibliographystyle{IEEEtran}
\bibliography{mybib}

\end{document}